\documentclass[twoside,leqno,10pt]{amsart}
\usepackage{amsfonts}
\usepackage{amsmath}
\usepackage{amscd}
\usepackage{amssymb}
\usepackage{amsthm}
\usepackage{amsrefs}
\usepackage{latexsym}
\usepackage{bbm}
\numberwithin{equation}{section}

\begin{document}
\newtheorem{conjecture}{Conjecture}
\newtheorem{theorem}{Theorem}
\newtheorem{lemma}{Lemma}
\newtheorem{corollary}{Corollary}
\numberwithin{equation}{section}
\newcommand{\dif}{\mathrm{d}}
\newcommand{\intz}{\mathbb{Z}}
\newcommand{\ratq}{\mathbb{Q}}
\newcommand{\natn}{\mathbb{N}}
\newcommand{\comc}{\mathbb{C}}
\newcommand{\rear}{\mathbb{R}}
\newcommand{\prip}{\mathbb{P}}
\newcommand{\uph}{\mathbb{H}}
\newcommand{\fief}{\mathbb{F}}
\newcommand{\majorarc}{\mathfrak{M}}
\newcommand{\minorarc}{\mathfrak{m}}
\newcommand{\sings}{\mathfrak{S}}

\title{On Hecke eigenvalues at Piatetski-Shapiro primes}
\date{\today}
\author{Stephan Baier \and Liangyi Zhao}
\maketitle

\begin{abstract}
Let $\lambda(n)$ be the normalized $n$-th Fourier coefficient of a holomorphic cusp form for the full modular group. We show that for some constant $C>0$ depending on the cusp form and every fixed $c$ in the range $1<c<8/7$, the mean value of $\lambda(p)$ is $\ll \exp(-C\sqrt{\log N})$ as $p$ runs over all (Piatetski-Shapiro) primes of the form $\left[n^c\right]$ with $n\in \mathbbm{N}$ and $n\le N$.
\end{abstract}



\section{Introduction}

Let $f$ be a holomorphic cusp form of weight $\kappa$ for the full modular group. By $\lambda_f(n)$ we denote the normalized $n$-th Fourier coefficient of $f$, {\it i.e.}
$$
f(z)=\sum\limits_{n=1}^{\infty} \lambda_f(n)n^{(\kappa-1)/2}e(nz)
$$
for $\Im z>0$.  The Ramanujan-Petersson conjecture, proved by P. Deligne \cite{Deli, Deli2}, states that $\lambda_f(n)\ll n^{\varepsilon}$ for any fixed $\varepsilon>0$. More precisely, we have
\begin{equation} \label{Ramanujanpetersson}
 \lambda_f(n) \ll d(n) \ll n^{\varepsilon},
\end{equation}
where $d(n)$ is the number of divisors of $n$.  If we assume, in addition, that $f$ is an eigenform of all the Hecke operators, then $f$ can be normalized such that $\lambda_f(1)=1$ and with this normalization the implied constant in the first ``$\ll$'' in \eqref{Ramanujanpetersson} can be taken to be 1. \newline

The distribution of Fourier coefficients of cusp forms has received a lot of attention. It is due to G. H. Hardy and S. Ramanujan that
\[ \sum_{n \leq N} \lambda_f (n) e ( \alpha n ) \ll N^{1/2} \log (2N), \]
and it follows from a general formula of K. Chandrasekharan and N. Narasimhan \cite{ChNa} that
\[ \sum_{n \leq N} \lambda_f (n) \ll N^{1/3+\varepsilon}. \]
It is worth noting that more recently N. J. E. Pitt \cite{Pitt} and V. Blomer \cite{Blo} established, respectively, estimates for sums of the forms
\[ \sum_{n \leq N} \lambda(n) e \left( \alpha n^2 + \beta n \right) \; \mbox{and} \; \sum_{n \leq N} \lambda \left( n^2 + s n + t \right), \]
with $\alpha, \beta \in \rear$ and $s, t \in \intz$. \newline

Especially interesting is the distribution of Fourier coefficients of cusp form at prime arguments. It is known that (see, for example, Section 5.6 of \cite{HIEK}) there exists a positive constant $C$, depending on the cusp form $f$, such that
\begin{equation} \label{coeffatprimes}
\sum\limits_{p\le N} \lambda_f(p) \ll N \exp \left( -C \sqrt{\log N} \right),
\end{equation}
where the implied $\ll$-constant depends on the cusp form $f$. Under the generalized Riemann hypothesis for modular $L$-functions, the right-hand side of \eqref{coeffatprimes} can be replaced by $N^{1/2+\varepsilon}$. M. R. Murty \cite{MRM} conjectured an $\Omega$-result of the form
\[ \sum_{p \leq N} \lambda_f (p) = \Omega_{\pm} \left( \frac{\sqrt{N} \log \log \log N}{\log N} \right) \]
and succeeded in proving it provided some $L$-function has no real zero between 1/2 and 1.  S. D. Adhikari \cite{SDA} generalized this result to cusp forms for the group $\Gamma_0(N)$.
The second-named author of the present paper investigated special exponential sums with Fourier coefficients of cusp forms over primes \cite{LZ1}, motivated by some surprising heuristic due to H. Iwaniec, W. Luo and P. Sarnak \cite{ILS} which gives that there should not be square-root cancellation in these exponential sums. \newline

There is a more precise conjecture than \eqref{coeffatprimes} on the distribution of the $\lambda_f(p)$'s, known as the Sato-Tate conjecture.  This conjecture states that if $f$ is a primitive holomorphic cusp form of weight greater than $2$ which is not of dihedral type, then the coefficients $\lambda_f(p)$ follow a certain distribution law. For the details, see \cite{HIEK}, Chapter 21. \newline

In the present paper, we investigate sums of Fourier coefficients of cusp forms over certain {\it sparse} sets of primes, namely Piatetski-Shapiro primes which we will discuss below. The motivation for our investigation is two-fold. First, the mean-values of arithmetic functions (in particular, of Fourier coefficients of cusp forms) over sparse sequences are often difficult to handle and thus of great interest.  The work of V. Blomer \cite{Blo} is in this direction. Second, it is a hard problem to detect primes in arithmetically interesting sets of natural numbers that are {\it sparse}.  Recently, there has been much progress with regard to problems of this type. J. B. Friedlander and H. Iwaniec \cite{FrIw} established the celebrated result that there are infinitely many primes of the form $X^2+Y^4$ with $X,Y \in \mathbbm{N}$.  D. R. Heath-Brown \cite{Hea1} proved the infinitude of the set of primes of the form $X^3+2Y^3$ with $X,Y\in \mathbbm{N}$.  \newline

A classical result in the direction of finding primes in sparse sequences is the Piatetski-Shapiro prime number theorem which states that there exists $c>1$ such that there are infinitely many primes of the form $\left[n^c\right]$ with $n\in \mathbbm{N}$, where $[ x ]$ denotes the integral part of $x$. More precisely, I. I. Piatetski-Shapiro \cite{Piat} proved that
\begin{equation} \label{piatshap}
\left|\{n\le N\ :\ \left[n^c\right] \mbox{ is a prime number}\}\right| \sim \frac{N}{c\log N}, \; \mbox{ as } N\rightarrow\infty
\end{equation}
if $c$ is a fixed number lying in the range $1<c<12/11$. This $c$-range for which \eqref{piatshap} holds has been widened by many authors (see \cites{Hea2, Klo1, Klo2, LiRi, Riv}). The most recent record is due to J. Rivat and P. Sargos \cite{RiSa} who proved that \eqref{piatshap} holds in the range $1<c< 2817/2426$. Lower bounds of the correct order of magnitude for the quantity on the left-hand side of \eqref{piatshap} were established by several authors (see \cites{Jia1, Jia2}) for wider $c$-ranges.  It is conjectured that the asymptotic formula in \eqref{piatshap} holds for all non-integers $c>1$. (Note that for $0 < c \leq 1$, \eqref{piatshap} follows easily from partial summation and the prime number theorem.) D. Leitman and D. Wolke \cite{LeWo} showed that \eqref{piatshap} holds for almost all, with respect to Lebesgue measure, $c$ with $1<c<2$.  Moreover, it was due to J.-M. Deshouillers \cite{Des} that the left-hand side of \eqref{piatshap} tends to infinity as $n$ tends to infinity for almost all $c >1$, with respect to Lebesgue measure.\newline

The main result of this paper is the following.

\begin{theorem} \label{mainresult}
Let $1<c<8/7$ and $\lambda_f(n)$ be the normalized $n$-th Fourier coefficient of a holomorphic cusp form $f$ for the full modular group.  Let $\mathbbm{P}$ denote the set of primes. Then there exists a constant $C>0$ depending on $f$ such that
\begin{equation} \label{themainresult}
\sum\limits_{\substack{n\le N\\ [n^c]\in \mathbbm{P}}} \lambda_f\left(\left[n^c\right]\right) \ll N\exp \left( -C \sqrt{\log N} \right),
\end{equation}
where the implied $\ll$-constant depends on $c$ and the cusp form $f$.
\end{theorem}

We note that, by \eqref{piatshap} and $8/7=1.142...<1.161...=2817/2426$ (recall that Sargos and Rivat established \eqref{piatshap} for the range $1<c<2817/2426$), the right-hand side of \eqref{themainresult} is small compared to the total number of Piatetski-Shapiro primes of the form $[n^c]$ with $n\le N$ if $c$ lies in the range given in Theorem \ref{mainresult}. \newline

An even harder problem is the question how the absolute values of the Fourier coefficients of cusp forms are distributed at Piatetski-Shapiro primes. For the full set of primes, one has the following results. If $f$ is a normalized Hecke eigenform, then, similarly as the prime number theorem, it can be established by using the analytic properties of the Rankin-Selberg $L$-function $L(f\otimes \overline{f},s)$ that
\begin{equation} \label{squaremean}
 \sum\limits_{p\le N} |\lambda_f(p)|^2 \sim \frac{N}{\log N}. 
\end{equation}
For general cusp forms $f$, one has
\begin{equation} \label{squaremeangen}
 \sum\limits_{p\le N} |\lambda_f(p)|^2 \sim c_f \frac{N}{\log N}
 \end{equation}
as $N\rightarrow \infty$, where $c_f$ is some positive constant depending on $f$. To see this, write $f$ as a linear combination of Hecke eigenforms and thus $\lambda_f(p)$ as a linear combination of the corresponding Fourier coefficients of these Hecke eigenforms, multiply out the modulus square, and use \eqref{squaremean} together with the similarly established fact that
\begin{equation} \label{}
 \sum\limits_{p\le N} \lambda_f(p)\overline{\lambda_g(p)} = o\left(\frac{N}{\log N}\right)
\end{equation}
if $f$ and $g$ are linearly independent Hecke eigenforms.   We conjecture that a result analogous to \eqref{squaremeangen} holds for Piatetski-Shapiro primes.

\begin{conjecture} \label{conj}
Under the assumptions of Theorem \ref{mainresult}, there exists a constant $c_f>0$ such that
\[
\sum\limits_{\substack{n\le N\\ [n^c]\in \mathbbm{P}}} \left|\lambda_f\left(\left[n^c\right]\right)\right|^2 \sim c_f\frac{N}{c\log N}, \; \mbox{ as } N\rightarrow\infty.
\]
\end{conjecture}

If Conjecture \ref{conj} holds, then, using \eqref{Ramanujanpetersson}, we deduce that
$$
\sum\limits_{\substack{n\le N\\ [n^c]\in \mathbbm{P}}} \left|\lambda_f\left(\left[n^c\right]\right)\right| \gg \frac{N}{\log N}
$$
which is large compared to the right-hand side of \eqref{themainresult}. This implies the following conditional result on the oscillations of Fourier coefficients of cusp forms at Piatetski-Shapiro primes.

\begin{theorem} \label{oscillations}
Assume that the conditions of Theorem \ref{mainresult} are satisfied and Conjecture \ref{conj} holds. Then either $\Re \lambda_f(p)$ or $\Im \lambda_f(p)$ changes sign infinitely often at primes of the form $p=[n^c]$, $n\in \mathbbm{N}$.
\end{theorem}

In the following, we say some words about our method for the proof of Theorem \ref{mainresult}. First, since every cusp form can be written as a linear combination of finitely many Hecke eigenforms, it will suffice to prove Theorem \ref{mainresult} for (normalized) Hecke eigenvalues. The advantages of working with Hecke eigenvalues are that they are multiplicative and real. Now for the proof of Theorem \ref{mainresult} with Hecke eigenvalues, we shall adapt parts of the method of \cite{LiRi} who established the validity of \eqref{piatshap} for $1<c<15/13=1.153...$. Similarly as in \cite{LiRi} ({\it c.f.} also the paper \cite{Uter} of the first-named author), we shall use estimates for certain trilinear exponential sums with monomials \cites{EFHI, RoSa}. However, the appearance of the Hecke eigenvalues shall require to introduce some new ingredients. In particular, we shall use a method of M. Jutila \cite{MJ} to transform exponential sums of the form
\[ \sum\limits_{N_1<n\le N_2} \lambda_f(n)e(g(n)) \]
into other exponential sums involving Hecke eigenvalues of different lengths. Jutila's method may be viewed as an analogue of the B-process in Weyl-van der Corput's method in the theory of exponential sums. \newline

We note that the investigations in this paper lead to exponential sums that are closely related to those considered in \cite{LZ1}. However, the method used in \cite{LZ1} will not be appropriate for our purposes here, and, conversely, the method used in the present paper does not seem to lead to any improvement of the result in \cite{LZ1}. \newline

{\bf Notations.} The following notations and conventions are used throughout the paper. \newline

\noindent $e(z) = \exp (2 \pi i z) = e^{2 \pi i z}$. \newline
$\eta$ and $\varepsilon$ are small positive real numbers, where $\varepsilon$ may not be the same number in each occurance. \newline
$c>1$ is a fixed number and we set $\gamma=1/c$. \newline
$\lambda(n)$ denotes the normalized $n$-th Fourier coefficients of a Hecke eigenform for the full modular group. \newline
$\Lambda(n)$ is the van Mangoldt function. \newline
$d(n)$ is the divisor function.\newline
$k\sim K$ means $K_1\le k \le K_2$ with $K/2\le K_1\le K_2\le 2K$. \newline
$f = O(g)$ or $f \ll g$ means $|f| \leq cg$ for some unspecified positive constant $c$. \newline
$f \asymp g$ means $f \ll g$ and $g \ll f$.\newline
$[x]$ denotes the largest integer not exceeding $x$, and $\psi(x)=x-[x]-1/2$ denotes the saw-tooth function.

\section{Preliminary Lemmas}

In this section, we quote the results needed later. To get started, we shall use the following approximation of the saw tooth function $\psi(x)=x-[x]-1/2$ due to J. D. Vaaler \cite{vaal}.

\begin{lemma} [Vaaler] \label{Valer}
For $0<|t|<1$, let
$$
W(t)=\pi t(1-|t|)\cot \pi t +|t|.
$$
Fix a positive integer $J$. For $x\in \mathbbm{R}$ define
$$
\psi^*(x):=-\sum\limits_{1\le |j|\le J} (2\pi ij)^{-1} W\left(\frac{j}{J+1}\right)e(jx)
$$
and
$$
\delta(x):=\frac{1}{2J+2}\sum\limits_{|j|\le J} \left(1-\frac{|j|}{J+1}\right)e(jx).
$$
Then $\delta$ is non-negative, and we have
$$
|\psi^*(x)-\psi(x)|\le \delta(x)
$$
for all real numbers $x$.
\end{lemma}

\begin{proof} This is Theorem A6 in \cite{GK} and has its origin in \cite{vaal}. \end{proof}

We shall also use the following estimate for a sum involving the function $\delta$.

\begin{lemma} \label{deltaest}
Fix $0<\gamma<1$. Assume that $1\le N< N_1\le 2N$. Define the function $\delta$ as in Lemma \ref{Valer}. Then
$$
\sum\limits_{N<n\le N_1} \delta\left(-n^{\gamma}\right)\ll
J^{-1}N+J^{1/2}N^{\gamma/2}.
$$
\end{lemma}

\begin{proof} This was proved on page 48 in \cite{GK}. \end{proof}

We shall also need the following variant of the prime number theorem for Hecke eigenvalues which is equivalent to \eqref{coeffatprimes}.

\begin{lemma} \label{cuspatprimes}
There exists a positive constant $C$, such that
\begin{equation*}
\sum\limits_{n\le N} \Lambda(n)\lambda(n) \ll N \exp \left( - C \sqrt{\log N} \right),
\end{equation*}
where the implied $\ll$-constant and the constant $C$ depend on the cusp form.
\end{lemma}

\begin{proof} This is a special case of the more general Theorem 5.12 in \cite{HIEK}. \end{proof}

We shall then see that it suffices to prove that

\begin{equation} \label{goal0}
\sum\limits_{n\sim N} \Lambda(n)f(n) =O\left(N^{1-\eta}\right)
\end{equation}
for a some fixed $\eta>0$, where $f$ is a certain function involving the Hecke eigenvalue $\lambda(n)$ and a trigonometric polynomial.  The following lemma reduces the above sum containing the von Mangoldt function to so-called type I and type II sums.

\begin{lemma} [Heath-Brown] \label{Heath}
Let $f$ be a complex-valued function defined on the natural numbers.  Suppose that $u$, $v$ and $z$ are real parameters satisfying the conditions
\[ 3\le u<v<z<2N,\ z-1/2\in \mathbbm{N},\ z\ge 4u^2,\ N\ge 32z^2u,\ v^3\ge 64N. \]
Suppose further that $1\le Y\le N$ and $XY=N$. Assume that $a_m$ and $b_n$ are complex numbers.  We write
\begin{equation} \label{Kdef}
 K:=\mathop{\sum_{m\sim X} \sum_{ n\sim Y}}_{mn\sim N} a_m f(mn)
\end{equation}
and
\begin{equation} \label{Ldef}
L:=\mathop{\sum_{m\sim X} \sum_{n\sim Y}}_{mn\sim N} a_m b_n f(mn).
\end{equation}
Then the estimate \eqref{goal0} holds if we uniformly have
\[ K\ll N^{1-2\eta}\ \  \mbox{ for } Y\ge z \; \mbox{and any complex} \; a_m\ll 1 \]
and
\[ L\ll N^{1-2\eta}\ \  \mbox{ for } u\le Y\le v \; \mbox{and any complex} \; a_m,b_n\ll 1. \]
\end{lemma}

\begin{proof} This is a consequence of Lemma 3 in \cite{Hea2}. \end{proof}

To separate the variables $m$ an $n$ appearing in the previous Lemma \ref{Heath}, we shall use the following lemmas. The first of them is the multiplicative property of Hecke eigenvalues, and the second of them is a variant of Perron's formula.

\begin{lemma} \label{multhecke}
Hecke eigenvalues are multiplicative and they satisfy the
following relation.
\begin{equation*}
\lambda (mn) = \sum_{d| \gcd(m,n)} \mu (d) \lambda \left( \frac{m}{d} \right) \lambda \left( \frac{n}{d} \right).
\end{equation*}
\end{lemma}

\begin{proof} This Lemma follows by applying the M\"{o}bius inversion formula to the product formula for the Hecke eigenvalues.  See, for example, Proposition 14.9 of \cite{HIEK}. \end{proof}

\begin{lemma} \label{Perron}
Let $0<M\le N<\nu N<\kappa M$ and let $a_m$ be complex numbers with $|a_m|\le 1$. We then have
\begin{equation} \label{perroneq}
\sum\limits_{N<n<\nu N} a_n = \frac{1}{2\pi} \int\limits_{-M}^M \left(\sum\limits_{M<m<\kappa M} a_m m^{-it}\right) N^{it}(\nu^{it}-1)t^{-1} dt \ + \ O(\log(2+M)),
\end{equation}
where the implied $O$-constant depends only on $\kappa$.
\end{lemma}

\begin{proof} This is Lemma 6 in \cite{EFHI}. \end{proof}

We shall be led to certain trilinear exponential sums with monomials. A part of them shall be estimated by using the following bound due to O. Robert and P. Sargos \cite{RoSa} which is a sharpening of an earlier estimate of E. Fouvry and H. Iwaniec \cite{EFHI}.

\begin{lemma} [Robert and Sargos] \label{monomials}
Let $\alpha$, $\alpha_1$, $\alpha_2$ be real constants such
that $\alpha\not=1$, $\alpha\alpha_1\alpha_2\not=0$. Let $M,M_1,M_2,x\ge 1$ and $\vert \phi_m\vert\le 1$, $\vert \psi_{m_1,m_2}\vert\le 1$.
Then we have
\begin{eqnarray*}
& &
\sum\limits_{m\sim M}\ \sum\limits_{m_1\sim M_1}\ \sum\limits_{m_2\sim M_2}
\phi_m \psi_{m_1,m_2} e\left(x
\frac{m^{\alpha}m_1^{\alpha_1}m_2^{\alpha_2}}
{M^{\alpha}M_1^{\alpha_1}M_2^{\alpha_2}}\right)\\
&\ll& \left(x^{1/4}M^{1/2}(M_1M_2)^{3/4}+M^{1/2}M_1M_2+M(M_1M_2)^{3/4}+
x^{-1/2}MM_1M_2\right) (MM_1M_2)^{\varepsilon}.
\end{eqnarray*}
\end{lemma}

\begin{proof} This follows from Theorem 1 in \cite{RoSa}. \end{proof}

To transform exponential sums of the form $\sum_n \lambda(n)e(g(n))$ into other exponential sums involving Hecke eigenvalues, we shall utilize Jutila's result \cite{MJ} quoted below.

\begin{lemma}[Jutila] \label{expsumtransformation}
Let $\delta_1,\delta_2,...$ denote positive constants which may be supposed to be arbitrarily small. Further, let $M_1\ge 2$ and put $L=\log M_1$. Let $a(n)$ be the $n$-th Fourier coefficient of a holomorphic cusp form $f$ for the full modular group, {\it i.e.}
$$
f(z)=\sum\limits_{n=1}^{\infty} a(n)e(nz)
$$
for $\Im z>0$. Let $\kappa$ be the weight of the cusp form $f$.
Let $M_1<M_2\le 2M_1$, and let $g$ and $w$ be holomorphic functions in the domain
$$
D=\{z\ :\ |z-x|<cM_1 \mbox{ for some } x\in [M_1,M_2]\},
$$
where $c$ is a positive constant. Suppose that $g(x)$ is real for $M_1\le x\le M_2$. Suppose also that, for some positive numbers $G$ and $W$,
\begin{equation} \label{wgcond}
|w(z)|\ll W \quad \mbox{and} \quad |g'(z)|\ll GM_1^{-1} \quad \mbox{for } z\in D,
\end{equation}
and that
\begin{equation} \label{gprprcond}
(0<) \; g''(x)\gg GM_1^{-2} \quad \mbox{for } M_1\le x\le M_2.
\end{equation}
Let $r=l/k$ with $(l,k)=1$ be a rational number such that
\begin{equation} \label{otto}
1\le k\ll M_1^{1/2-\delta_1}, \quad |r| \asymp GM_1^{-1} \quad \mbox{and} \quad g'(M_0)=r
\end{equation}
for a certain number $M_0\in (M_1,M_2)$. Write
\begin{equation} \label{mjdef}
M_j=M_0+(-1)^jm_j,\quad j=1,2.
\end{equation}
Suppose that $m_1\asymp m_2$, and that
\begin{equation} \label{M1M2cond}
M_1^{\delta_2}\max\left\{M_1G^{-1/2},|lk|\right\}\ll m_1\ll M_1^{1-\delta_3}.
\end{equation}
Define for $j=1,2$
$$
p_{j,n}=g(x)-rx+(-1)^{j-1}\left(\frac{2\sqrt{nx}}{k}-\frac{1}{8}\right) \quad \mbox{and} \quad n_j=(r-g'(M_j))^2k^2M_j,
$$
and for $n<n_j$ let $x_{j,n}$ be the (unique) zero of $p_{j,n}'(x)$ in the interval $(M_1,M_2)$. Then
\begin{equation} \label{jutilaeq}
\begin{split}
 \sum\limits_{M_1\le m\le M_2} & a(m)w(m)e(g(m)) \\
 & = \frac{i}{\sqrt{2k}}\sum\limits_{j=1}^2 (-1)^{j-1}\sum\limits_{n<n_j} a(n)e\left(-\frac{n\overline{l}}{k}\right)n^{-\kappa/2+1/4} x_{j,n}^{\kappa/2-3/4}\frac{w(x_{j,n})}{\sqrt{p_{j,n}''(x_{j,n})}} e \left( p_{j,n}(x_{j,n})+\frac{1}{8} \right) \\
& \hspace*{1cm}+ O\left(W(|l|k)^{1/2}M_1^{(\kappa-1)/2}m_1^{1/2}L^2+G^{1/2}W|l|^{-3/4}k^{5/4}M_1^{(\kappa-1)/2}m_1^{-1/4}L\right),
\end{split}
\end{equation}
where $l\overline{l}\equiv 1$ mod $k$.
\end{lemma}

\begin{proof} This is Theorem 3.2 in \cite{MJ} with different notations. \end{proof}

Lemma~\ref{expsumtransformation} lies at the heart of our method.  Using this result, Jutila \cite{MJ} proved the following estimate for ``long'' exponential sums with Hecke eigenvalues.

\begin{lemma}[Jutila] \label{expsumhecke}
Let $t\ge 1$ and $M\ge 2$. Assume that
\begin{equation} \label{Mtcond}
M^{3/4-\gamma}\ll t \ll M^{3/2-\gamma}.
\end{equation}
Then
\begin{equation} \label{jutilaeq2}
\sum\limits_{n\sim M} \lambda(n)e(tn^{\gamma})\ll t^{1/3}M^{1/2+\gamma/3}(tM)^{\varepsilon}.
\end{equation}
\end{lemma}

\begin{proof} This follows from  Theorem 4.6 in \cite{MJ}. \end{proof}

The above Lemma \ref{expsumhecke} is not needed in our method. We shall rather use Lemma  \ref{expsumtransformation} on short exponential sums with Hecke eigenvalues. However, in section \ref{splittingsection} we will see that a direct application of Lemma \ref{expsumhecke} also leads to a non-trivial result. Yet, this result is weaker than our main result, Theorem \ref{mainresult}. \newline

We also use the following lemma in the investigation of the spacing of certain monomial points.

\begin{lemma} \label{spacinglemma}
 Let $\alpha \beta \neq 0$, $\Delta > 0$, $M \geq 1$ and $N \geq 1$.  Let $\mathcal{A} (M,N;\Delta)$ be the number of quadruples $(m, \tilde m, n, \tilde n)$ such that
\[ \left| \left( \frac{\tilde m}{m} \right)^{\alpha} - \left( \frac{\tilde n}{n} \right)^{\beta} \right| < \Delta , \]
with $M \leq m, \tilde m < 2M$ and $N \leq n, \tilde n < 2N$.  We then have
\[ \mathcal{A}(M,N;\Delta) \ll MN \log 2MN + \Delta M^2N^2. \]
\end{lemma}

\begin{proof}
 This is Lemma 1 in \cite{EFHI}.
\end{proof}

Finally, we need the following classical exponential sum estimate.

\begin{lemma}[Van der Corput] \label{vandercorput}
Let $b-a\ge 1$ and let $f(x)$ be a twice differentiable function on $(a,b)$ such that $\Lambda\le |f''(x)| \le \nu \Lambda$, where $\Lambda>0$ and $\nu\ge 1$. Then
\[ \sum\limits_{a<n\le b} e(f(n)) \ll \nu \Lambda^{1/2}(b-a)+\Lambda^{-1/2}.\]
\end{lemma}

\begin{proof} This is Lemma 4.1 in \cite{MJ}.
\end{proof}

\section{Reduction to exponential sums}

Using \eqref{Ramanujanpetersson}, partial summation, and the fact that every cusp form can be written as a linear combination of finitely many Hecke eigenforms, Theorem \ref{mainresult}, our main result, can be easily deduced from the following result whose proof will be the object of the remainder of this paper. \newline

\begin{theorem} \label{mainresultmodified}
Let $1<c<8/7$ and $\lambda(n)$ be the normalized $n$-th Fourier coefficient of a Hecke eigenform for the full modular group. By $\Lambda(n)$ denote the von Mangoldt function. Then there exists a positive constant $C$ depending on the cusp form such that
\begin{equation} \label{goal}
\sum\limits_{n\le N} \Lambda\left(\left[n^c\right]\right)\lambda\left(\left[n^c\right]\right) \ll
N\exp(-C\sqrt{\log N}),
\end{equation}
where the implied $\ll$-constant depends only on $c$, $C$ and the cusp form.
\end{theorem}

In this section, we reduce the left-hand side of \eqref{goal} to  exponential sums.  Throughout the sequel, let $\gamma=1/c$. Then $\left[n^c\right]=m$ is equivalent to
$$
-(m+1)^{\gamma}< -n\le -m^{\gamma}.
$$
Therefore, we have
\begin{equation} \label{rewritten}
\sum\limits_{n\le N} \Lambda\left(\left[n^c\right]\right)\lambda\left(\left[n^c\right]\right) =\sum\limits_{m\le N^c} \left(\left[-m^{\gamma}\right]-\left[-(m+1)^{\gamma}
\right]\right)\Lambda(m)\lambda(m)+O(\log N).
\end{equation}
Breaking into dyadic intervals, it hence suffices to prove that
\begin{equation} \label{S}
S:=\sum\limits_{n\sim N^c} \left(\left[-n^{\gamma}\right]-\left[-(n+1)^{\gamma}
\right]\right)\Lambda(n)\lambda(n)\ll N\exp(-C\sqrt{\log N})
\end{equation}
for any $N>1$. We write the above sum $S$ in the form
\begin{equation} \label{S12}
S=S_1+S_2,
\end{equation}
where
$$
S_1=\sum\limits_{n\sim N^c} \left((n+1)^{\gamma}-n^{\gamma}\right)\Lambda(n)\lambda(n)
$$
and
$$
S_2=\sum\limits_{n\sim N^c} \left(\psi\left(-(n+1)^{\gamma}\right)-\psi\left(-n^{\gamma}\right)\right)\Lambda(n)\lambda(n),
$$
with $\psi(n)$ being the saw-tooth function in Lemma~\ref{Valer}.  Using partial summation and the bounds
$$
(x+1)^{\gamma}-x^{\gamma}\ll x^{\gamma-1} \quad \mbox{and} \quad
\frac{\dif}{\dif x} \left((x+1)^{\gamma}-x^{\gamma}\right) \ll x^{\gamma-2}
$$
for $x\ge 1$,
we deduce from Lemma \ref{cuspatprimes} that
$$
S_1\ll N\exp(-C\sqrt{\log N}),
$$
where the implied $\ll$-constant depends only on $\gamma$, $C$ and the cusp form. Our treatment of the sum $S_2$ begins like in \cite{GK}. By Lemma \ref{Valer}, we have the following. For any $J>0$ there exist functions $\psi^*$ and $\delta$, with $\delta$ non-negative, such that
$$
\psi(x)=\psi^*(x)+O( \delta(x) ),
$$
where
$$
\psi^*(x)=\sum\limits_{1\le |j|\le J} a(j)e(jx), \quad
\delta(x)=\sum\limits_{|j|\le J} b(j)e(jx)
$$
with
$$
a(j)\ll j^{-1}, \quad b(j)\ll J^{-1}.
$$
Consequently,
\begin{eqnarray*}
S_2&=&\sum\limits_{n\sim N^c} \left(\psi^*\left(-(n+1)^{\gamma}\right)-\psi^*\left(-n^{\gamma}\right)\right)\Lambda(n)\lambda(n)+
O\left((\log N) \sum\limits_{n\sim N^c} \left(\delta\left(-(n+1)^{\gamma}\right)+\delta\left(-n^{\gamma}\right)\right)  \right)\\
&=& S_3+O(S_4),
\end{eqnarray*}
say. We fix a small $\eta>0$ and set
\begin{equation} \label{J}
J:=N^{c-1+\eta}.
\end{equation}
Then, using Lemma \ref{deltaest}, we obtain
$$
S_4\ll N^{1-\eta/2}
$$
if $1<c<2$. \newline

The remaining task is to prove that
$$
S_3\ll N^{1-\eta/2},
$$
provided that $\eta$ is sufficiently small. We write
$$
S_3=\sum\limits_{1\le |j|\le J} \sum\limits_{n\sim N^c}  \Lambda(n)\lambda(n)a(j)\phi_j(n)e(-jn^{\gamma}),
$$
where $\phi_j(x)=1-e(j(x^{\gamma}-(x+1)^{\gamma}))$. Using partial summation and the bounds
$$
\phi_j(x)\ll jx^{\gamma-1} \quad \mbox{and} \quad \frac{\dif}{\dif x} \phi_j(x)\ll jx^{\gamma-2},
$$
we deduce that it suffices to prove that
$$
\sum\limits_{1\le |j|\le J} \left|\sum\limits_{n\sim N^c}  \Lambda(n)\lambda(n)e(-jn^{\gamma})\right| \ll N^{c-\eta/2}.
$$
Replacing $N^{c}$ by $N$, taking the definition of $J$ in \eqref{J} into account, dividing the summation interval $1\le |j|\le J$ into $O(\log 2J)$ dyadic intervals, and using the facts that $e(-x)=\overline{e(x)}$ and the Hecke eigenvalues are real, we see that the above bound holds if
\begin{equation} \label{goal1}
\sum\limits_{h\sim H} \left| \sum\limits_{n\sim N} \Lambda(n)\lambda(n)e\left(hn^{\gamma}\right)\right| \ll N^{1-\eta}
\end{equation}
for any $N\ge 1$ and $1\le H\le N^{1-\gamma+\eta}$. The following lemma reduces the term on the left-hand side of \eqref{goal1} to trilinear exponential sums.

\begin{lemma} \label{bilinearsums}
Suppose that $u$, $v$ and $z$ are real parameters satisfying the conditions
\begin{equation} \label{uvzcond}
3\le u<v<z<2N,\ z-1/2\in \mathbbm{N},\ z\ge 4u^2,\ N\ge 32z^2u,\ v^3\ge 64N. \end{equation}
Suppose further that $1\le Y\le N$, $XY=N$ and $H\ge 1$. Assume that $A_m$, $B_n$ and $C_h$ are complex numbers.  For $d\in \mathbbm{N}$ set
\begin{equation} \label{Kddef}
K_d:=\mathop{\sum\limits_{m\sim X/d}\ \sum_{ n\sim Y/d}}_{mn\sim N/d^2}\ \sum\limits_{h\sim H} \
A_m C_h \lambda(n)e\left(hd^{2\gamma}m^{\gamma}n^{\gamma}\right)
\end{equation}
and
\begin{equation} \label{Lddef}
L_d:=\mathop{\sum\limits_{m\sim X/d}\ \sum_{n\sim Y/d}}_{mn\sim N/d^2}\ \sum\limits_{h\sim H}\  A_{m} B_n C_h e\left(hd^{2\gamma}m^{\gamma}n^{\gamma}\right).
\end{equation}
Then the estimate \eqref{goal1} holds if we uniformly have
\begin{equation} \label{Kdbound}
K_d\ll N^{1-3\eta}d^{-1}\ \  \mbox{ for } Y\ge z,\ d\le 2Y \; \mbox{and any complex} \; A_m,C_h\ll 1
\end{equation}
and
\begin{equation} \label{Ldbound}
L_d\ll N^{1-3\eta}d^{-1}\ \  \mbox{ for } u\le Y\le v,\ d\le 2Y \; \mbox{and any complex} \; A_m,B_n,C_h\ll 1.
\end{equation}
\end{lemma}

\begin{proof} We first write
$$
\sum\limits_{h\sim H} \left| \sum\limits_{n\sim N} \Lambda(n)\lambda(n)e\left(hn^{\gamma}\right)\right| =
\sum\limits_{h\sim H} c_h \sum\limits_{n\sim N} \Lambda(n)\lambda(n)e\left(hn^{\gamma}\right),
$$
where $c_h$ are suitable complex numbers with $|c_h|=1$. We further set
$$
f(n)=\lambda(n)\sum\limits_{h\sim H} c_he\left(hn^{\gamma}\right)
$$
so that
$$
\sum\limits_{h\sim H} \left| \sum\limits_{n\sim N}
\Lambda(n)\lambda(n)e\left(hn^{\gamma}\right)\right|=\sum\limits_{n\sim N}
\Lambda(n)f(n).
$$
Now, by Lemma \ref{Heath}, the bound \eqref{goal1} holds if
\begin{equation} \label{KL}
K\ll N^{1-2\eta}\quad \mbox{and} \quad L\ll N^{1-2\eta}
\end{equation}
under the conditions of the same lemma. Here $K$ and $L$ are defined as in \eqref{Kdef} and \eqref{Ldef}. We may rewrite these terms in the form
$$
K=\mathop{\sum\limits_{m\sim X} \sum_{n\sim Y}}_{mn\sim N} \sum\limits_{h\sim H} a_mc_h\lambda(mn)e\left(h(mn)^{\gamma}\right)
$$
and
$$
L=\mathop{\sum\limits_{m\sim X} \sum_{n\sim Y}}_{mn\sim N} \sum\limits_{h\sim H} a_m b_nc_h\lambda(mn)e\left(h(mn)^{\gamma}\right).
$$
Using the multiplicative property of Hecke eigenvalues, Lemma \ref{multhecke}, we have
\begin{equation} \label{rewriteK}
K=\sum\limits_{d\le 2Y} \mu(d) \mathop{\sum\limits_{m\sim X/d}\ \sum_{n\sim Y/d}}_{mn\sim N/d^2}\ \sum\limits_{h\sim H} \
a_{dm} \lambda(m) c_h \lambda(n) e\left(hd^{2\gamma}m^{\gamma}n^{\gamma}\right)
\end{equation}
and
\begin{equation} \label{rewriteL}
L=\sum\limits_{d\le 2Y} \mu(d) \mathop{\sum\limits_{m\sim X/d}\ \sum_{ n\sim Y/d}}_{mn\sim N/d^2}\ \sum\limits_{h\sim H} \
a_{dm} \lambda(m) b_{dn} \lambda(n) c_h e\left(hd^{2\gamma}m^{\gamma}n^{\gamma}\right).
\end{equation}
Now, \eqref{KL} follows from \eqref{Kdbound}, \eqref{Ldbound},\eqref{rewriteK}, \eqref{rewriteL} and the bound
$\lambda(n)\ll n^{\varepsilon}$, the Ramanujan-Petersson conjecture, proved by Deligne \cite{Deli}.
 \end{proof}

In the following sections, we shall estimate the terms $K_d$ and $L_d$.

\section{Estimation of $L_d$}
In this section, we estimate $L_d$ defined in \eqref{Lddef}.

\begin{lemma} \label{QLdlemma}
Let $Q$ be any positive integer and $\varepsilon$
be any positive real number.
Then
\begin{equation} \label{QLd}
 \vert L_d\vert^2 \ll \left(QX\left(HX^{\gamma}Y^{\gamma}Q^{-1}\right)^{1/2} \left(H^2Q^{-1}Y^2+HY\right)+QX^{2-\gamma} \left(HY^{2-\gamma}+HYX^{\gamma}\right)\right)d^{-2}N^{\varepsilon}, \end{equation}
 where the implied $\ll$-constant depends only on $\varepsilon$.
\end{lemma}

\begin{proof} This lemma follows after a slight modification of the estimations in section 4 of \cite{Hea2}. \end{proof}

From Lemma \ref{QLdlemma}, we deduce the following result.

\begin{lemma} \label{Ldlemma}
If $H,N\ge 1$, $1\le Y\le 2N$ and $1\le d\le 2Y$, then
\begin{equation} \label{Ld}
\begin{split}
\vert L_d \vert^2 \ll \Large( N^{1+\gamma/2}H^{3/2}&+N^{2-\gamma}H+N^2HY^{-1}+
N^{4/3+\gamma/3}H^{2}Y^{1/3}  \\
& +N^{2/3+2\gamma/3}H^2Y^{2/3}+
N^{4/3-2\gamma/3}H^2Y^{4/3} \Large) d^{-2}N^{\varepsilon},
\end{split}
\end{equation}
where the implied $\ll$-constant depends only on $\varepsilon$.
\end{lemma}

\begin{proof} To optimize the estimate \eqref{QLd} in Lemma \ref{QLdlemma}, we choose
$$
Q:=1+\left[HX^{(\gamma-2)/3}Y^{(\gamma+2)/3}\right].
$$
Then \eqref{Ld} follows from \eqref{QLd} and $XY=N$ by a short calculation.
\end{proof}

Lemma \ref{Ldlemma} brings us into a position to formulate a condition under which the desired estimate $L_d\ll N^{1-3\eta}d^{-1}$ holds.

\begin{lemma} \label{Ldcondlemma}
For every sufficiently small fixed $\eta>0$, we have
\[ L_d\ll N^{1-3\eta}d^{-1}, \]
provided that $\gamma>5/6$, $1\le H\le N^{1-\gamma+\eta}$, $1\le d\le 2Y$ and
\begin{equation} \label{LYcond}
N^{1-\gamma+100\eta}\le Y\le N^{5\gamma-4-100\eta}.
\end{equation}
\end{lemma}

\begin{proof}
This Lemma follows from Lemma \ref{Ldlemma} by a short calculation and is analogous to Lemma 4 in \cite{Hea2}.
\end{proof}

\section{Estimation of $K_d$, first method}

We now establish some estimates for $K_d$, defined in \eqref{Kddef}, which are favorable if $Y$ is not too large. In this case, we ignore the special nature of the Hecke eigenvalues $\lambda(n)$ appearing in the sum $K_d$ and treat $K_d$ as a trilinear sum with arbitrary coefficients, like $L_d$ in the previous section. \newline

For $Y$ of medium size, we use the following result.

\begin{lemma} \label{Kdcondlemma0}
For every sufficiently small fixed $\eta>0$, we have
\[ K_d\ll N^{1-3\eta}d^{-1}, \]
provided that $\gamma>5/6$, $1\le H\le N^{1-\gamma+\eta}$, $1\le d\le 2Y$ and
\begin{equation} \label{KYcond}
N^{5-5\gamma+100\eta}\le Y\le N^{\gamma-100\eta}.
\end{equation}
\end{lemma}

\begin{proof}
This can be proved in essentially the same way as Lemma \ref{Ldcondlemma}, where the roles of $X$ and $Y$ are reversed. Similarly as in Lemma \ref{Ldcondlemma}, we get that $K_d\ll N^{1-3\eta}d^{-1}$, provided that
\[
N^{1-\gamma+100\eta}\le X\le N^{5\gamma-4-100\eta}.\]
This is equivalent to \eqref{KYcond} since $XY=N$.
\end{proof}

If $Y$ is small, then, similarly as in \cite{LiRi}, we can directly apply Lemma \ref{monomials} to estimate the term $K_d$ defined in \eqref{Kddef}. This gives the following result.

\begin{lemma} \label{estimateforKd}
If $H,N\ge 1$, $N^{\eta}H\le Y\le 2N$ and $1\le d\le 2Y$, then
\begin{equation} \label{Kdest1}
K_d \ll d^{-1}\left(N^{3/4+\gamma/4}HY^{-1/4}+ NHY^{-1/2}+N^{3/4}H^{3/4}Y^{1/4}+ N^{1-\gamma/2}H^{1/2}\right)(NH)^{\varepsilon}.
\end{equation}
\end{lemma}

\begin{proof}
First, we remove the summation condition $mn\sim N/d^2$ on the right-hand side of \eqref{Kddef} by using Lemma \ref{Perron} and thus make the summation ranges of $m$ and $n$ independent. After applying the bound $\lambda(n)\ll n^{\varepsilon}$ (Ramanujan-Petersson conjecture), the first term on the right-hand side of \eqref{perroneq} leads to expressions of the form
$$
N^{\varepsilon}\sum\limits_{m\sim X/d}\
\sum\limits_{n\sim Y/d}\ \sum\limits_{h\sim H}\phi_m\psi_n\epsilon_h
e\left(h d^{2\gamma} m^{\gamma}n^{\gamma}\right)
$$
with $|\phi_m|,|\psi_n|,|\epsilon_h|\le 1$.  We then estimate these trilinear sums by applying Lemma~\ref{monomials} with the following choice of parameters:
$$
x:=HN^{\gamma}, \; M:= \frac{Y}{d},\; M_1:=\frac{X}{d},\; M_2:=H, \; \alpha:=\gamma,\; \alpha_1:=\gamma,\; \alpha_2:=1.
$$
Additionally taking $XY=N$ into account, we arrive at the estimate \eqref{Kdest1} upon noting that the contribution of the $O$-term on the right-hand side of \eqref{perroneq} is $\ll N^{1+\varepsilon}HY^{-1}d^{-1}$ and thus negligible by the condition $Y\ge N^{\eta}H$ in the lemma. \end{proof}

Lemma \ref{estimateforKd} enables us to formulate another condition under which the desired estimate $K_d\ll N^{1-3\eta}d^{-1}$ holds.

\begin{lemma} \label{Kdcondlemma}
For every sufficiently small fixed $\eta>0$, we have
\[ K_d\ll N^{1-3\eta}d^{-1}, \]
provided that $\gamma>5/6$, $1\le H\le N^{1-\gamma+\eta}$, $1\le d\le 2Y$ and
$$
N^{3-3\gamma+100\eta}\le Y\le N^{3\gamma-2-100\eta}.
$$
\end{lemma}

\begin{proof}
This follows from Lemma \ref{estimateforKd} by a short calculation.
\end{proof}

We note that Lemma \ref{Kdcondlemma} could also be established by using the original bound of E. Fouvry and H. Iwaniec for trilinear exponential sums with monomials, Theorem 3 in \cite{EFHI}, but with the more restrictive condition $\gamma>17/20$ in place of $\gamma>5/6$. Moreover, we would need to apply Theorem 3 in \cite{EFHI} twice with different choices of $M$, $M_1$ and $M_2$ which would complicate the computations.

\section{Splitting of the sum $K_d$ \label{splittingsection}}

We now turn to the case when $Y$ is large in which the special nature of the Hecke eigenvalues $\lambda(n)$ will become important. \newline

A natural idea would be to apply Jutila's estimate for ``long'' exponential sums with Hecke eigenvalues in Lemma \ref{expsumhecke} directly to the sum over $n$ in $K_d$ and then to sum up over $h$ and $m$ trivially. But it turns out that this leads to the condition $Y\gg N^{8/3-2\gamma+100\eta}$ which is not sufficient to establish the $c$-range $1<c<8/7$ in Theorem~\ref{mainresult}. In order to obtain the desired estimate $K_d\ll N^{1-3\eta}d^{-1}$ for all relevant $Y$'s, we would need that $\gamma>8/9$ which means that we would only get the $c$-range $1<c<9/8=1.125$ in Theorem \ref{mainresult}. This is due to the fact that $N^{8/3-2\gamma+100\eta}$ is larger than the term $N^{\gamma-100\eta}$ in \eqref{KYcond} whenever $\gamma\le 8/9$. \newline

To obtain the desired estimate for $K_d$ in an as large as possible $Y$-range, we proceed as follows. First, following Jutila \cite{MJ}, we split the sum involving Hecke eigenvalues over $n$ into shorter sums which we then transform into new exponential sums with Hecke eigenvalues by applying Lemma~\ref{expsumtransformation} due to Jutila.  Collecting all terms, we arrive at multi-linear exponential sums. To estimate them,  we refine Jutila's treatment of long exponential sums with Hecke eigenvalues in \cite{MJ}. Here we take advantage of the additional summations over $h$ and $m$. This will lead to a spacing problem with certain points depending on $h$, $m$ and further integers. We shall show that these points are essentially distributed as expected. \newline

We first make some observations on Farey sequences.  Let $K \geq 1$.  By $\mathcal{F}(K)$, we denote the extended sequence of Farey fractions of level $K$ consisting of all fractions of the form $l/k$, $1 \leq k \leq K$, $\gcd(l,k)=1$.  For two consecutive Farey fractions $l/k$ and $l'/k'$ in the sequence $\mathcal{F}(K)$, define the mediant, $\rho$, to be
\begin{equation} \label{mediantdef}
\rho \left( \frac{l}{k}, \frac{l'}{k'} \right) = \frac{l+l'}{k+k'}.
\end{equation}
Furthermore, if $l''/k'' < l/k <l'/k'$ are three consecutive Farey fractions, then we define the Farey interval (depending on $K$) around $l/k$ by
\begin{equation} \label{fareyintdef}
I \left( \frac{l}{k} \right) = \left( \rho \left( \frac{l''}{k''}, \frac{l}{k} \right) , \rho \left( \frac{l}{k}, \frac{l'}{k'} \right) \right].
\end{equation}
We note that the set of the real numbers is the disjoint union of all these Farey intervals.  Under the above notations, we further define $A_j(l/k)$ for $j=1,2$ by
\[ \left( \frac{l}{k}-\frac{A_1(l/k)}{kK} , \frac{l}{k}+\frac{A_2(l/k)}{kK} \right] = I \left( \frac{l}{k} \right). \]
We note that (see, for example, (4.2.11) in \cite{MJ})
\begin{equation} \label{ajest}
A_j \left( \frac{l}{k} \right) \asymp 1.
\end{equation}

For the proof of the desired bound $K_d \ll N^{1-3\eta}/d$, it suffices, by partial summation, the realness of $\lambda(n)$ and the fact $e(-x)= \overline{e(x)}$, to prove that
\begin{equation} \label{tildeKest}
\tilde K_d \ll \frac{N^{1-3\eta}}{d} \left( \frac{Y}{d} \right)^{\frac{\kappa-1}{2}},
\end{equation}
where
\begin{equation} \label{tildeKdef}
\tilde K_d:= \sum\limits_{h\sim H} \sum\limits_{m\sim X/d} \sum_{\substack{ Y_1/d < n\leq Y_2/d \\ N_1/(d^2m) < n \leq N_2/(d^2m)}}  A_m C_h a(n) e\left(-hd^{2\gamma}m^{\gamma}n^{\gamma}\right).
\end{equation}
Here
\[ a(n) = \lambda(n) n^{\frac{\kappa-1}{2}} \]
is the un-normalized Fourier coefficient of the cusp form and
\[ Y_j \asymp Y \; \mbox{and} \; N_j \asymp N, \; \mbox{for} \; j=1,2. \]
We prefer to have a negative sign in the $e$-term on the right-hand side of \eqref{tildeKdef} for technical reasons. \newline

We describe here briefly what we will do in the remainder of the section.  We shall split the sum $\tilde K_d$ into a sum of short exponential sums to which a Jutila-type transformation lemma (Lemma~\ref{transformlemma}) can be applied.  To this end, we cut the summation over $n$ into small pieces so that the value of the derivative of the amplititude function in \eqref{tildeKdef} on each of the small pieces is close to a fraction $l/k$ whose denominator is not too large.  In this treatment, we may incur an error which comes from the possible imperfect fit of the ``end-intervals'' in the splitting.  This error will be estimated. \newline

For $d \in \natn$, $h \sim H$, $m \sim X/d$, $l <0$ and $1 \leq k \leq K$, we define
\begin{equation} \label{ghmddef}
 g_{d,h,m}(x) = -hd^{2\gamma} m^{\gamma} x^{\gamma}
\end{equation}
and $M_j(d,h,m; l/k)$ for $j=0,1,2$ by
\begin{equation} \label{Mjdhmlkdef}
g' \left( M_0 \left(d,h,m; \frac{l}{k} \right) \right) = \frac{l}{k}, \; \mbox{and} \; g' \left( M_j \left(d,h,m; \frac{l}{k} \right) \right)= \frac{l}{k} + (-1)^j \frac{A_j(l/k)}{kK} .
\end{equation}
We further set
\[ \mathcal{J}\left(d,h,m; \frac{l}{k} \right) = \left( M_1 \left(d,h,m; \frac{l}{k} \right) ,  M_2 \left(d,h,m; \frac{l}{k} \right) \right] \]
and
\[ \mathcal{I}(d, m) = \left( \frac{Y_1}{d}, \frac{Y_2}{d} \right] \bigcap \left( \frac{N_1}{d^2m}, \frac{N_2}{d^2m} \right] . \]

As mentioned above, we shall approximate $\tilde K_d$ by
\begin{equation} \label{Kstardef}
K^*_d = \sum_{h \sim H} \sum_{m \sim X/d} C_h A_m \sum_{1 \leq k \leq K} \sum_{\substack{l<0,\ \gcd(l,k) =1 \\ M_0(d, h,m;l/k) \in \mathcal{I}(d,m) }} \sum_{n \in \mathcal{J}(d,h,m;l/k)} a(n) e (-hd^{2\gamma} m^{\gamma} n^{\gamma}).
\end{equation}
We now estimate the error of this approximation.  For every $h$ and $m$, there are at most two fractions of the form $l/k$ with $1 \leq k \leq K$ and $\gcd (l,k)=1$ such that the interval $\mathcal{J}(d,h,m;l/k)$ is not contained in the interval $\mathcal{I}(d, m)$
but over-laps with $\mathcal{I}(d, m)$.  The contribution arising from an interval $\mathcal{J}(d,h,m;l/k)$ of this kind to the inner-triple sum on the right-hand side of \eqref{Kstardef} is
\[ \ll \left( \frac{Y}{d} \right)^{\frac{\kappa-1}{2}+\varepsilon} \left| \mathcal{J}\left(d,h,m; \frac{l}{k} \right) \right| . \]
Using \eqref{ajest}, it is easy to compute, with $h$, $m$, $k$ and $l$ subject to the same conditions as those in the summations in \eqref{Kstardef}, that the length of $\mathcal{J}(d,h,m;l/k)$ is
\begin{equation} \label{Jlength}
 \left| \mathcal{J}\left(d,h,m; \frac{l}{k} \right) \right| \ll \frac{Y^{2-\gamma}}{Hd^2 X^{\gamma} kK}
\end{equation}
(compare with \eqref{firstobserve}).  Thus, the error in approximating $\tilde K_d$ by $K^*_d$ is
\begin{equation} \label{approxerror}
\tilde K_d - K^*_d \ll H \frac{X}{d} \left( \frac{Y}{d} \right)^{\frac{\kappa-1}{2} +\varepsilon} \frac{Y^{2-\gamma}}{Hd^2 X^{\gamma} K} \ll N^{1-\gamma+\varepsilon} \frac{Y}{K d^3} \left( \frac{Y}{d} \right)^{\frac{\kappa-1}{2}},
\end{equation}
where we use $XY=N$.  This above error is negligible, i.e.
\begin{equation} \label{errorneg}
 \tilde K_d - K^*_d \ll \frac{N^{1-3\eta}}{d} \left( \frac{Y}{d} \right)^{\frac{\kappa-1}{2}},
\end{equation}
if
\begin{equation} \label{Kcondition}
 K \gg N^{4\eta} \frac{Y}{N^{\gamma}d^2}.
\end{equation}
After a short computation, $K^*_d$ can be further simplified into
\begin{equation} \label{Kstareq}
K^*_d = \sum_{h \sim H} \sum_{m \sim X/d} C_h A_m \sum_{1 \leq k \leq K} \sum_{\substack{l<0\\ \gcd(l,k) =1 \\ |l| \in \mathcal{L}(d,h,m,k) }} \sum_{n \in \mathcal{J}(d,h,m;l/k)} a(n) e (-hd^{2\gamma} m^{\gamma} n^{\gamma}),
\end{equation}
where
\[ \mathcal{L}(d, h, m, k) = \left[ \frac{hd^{1+\gamma}m^{\gamma}\gamma k}{Y_2^{1-\gamma}}, \frac{hd^{1+\gamma}m^{\gamma}\gamma k}{Y_1^{1-\gamma}} \right) \bigcap \left[ \frac{hd^2m\gamma k}{N_2^{1-\gamma}}, \frac{hd^2m\gamma k}{N_1^{1-\gamma}} \right). \]
In the following sections, we shall show that
\begin{equation} \label{kstartdgoal}
 K^*_d \ll \frac{N^{1-3\eta}}{d} \left( \frac{Y}{d} \right)^{(\kappa-1)/2},
\end{equation}
with a suitably chosen parameter $K$.

\section{Transformation of short exponential sums with Hecke eigenvalue coefficients}

In this section, we transform the short exponential sum
\begin{equation} \label{startpoint}
\sum_{M_1\leq m \leq M_2} a(m) e (-tm^{\gamma})
\end{equation}
using Lemma~\ref{expsumtransformation}.  Note that the inner-most sum of \eqref{Kstareq} is of this form.  We have the following

\begin{lemma} \label{transformlemma}
Let $\delta_1,\delta_2,...$ denote positive constants which may be supposed to be arbitrarily small.  Let $t>0$, $K \geq 1$, $A_1 \asymp 1$, $A_2 \asymp 1$, $l, k \in \intz$ with $l< 0$, $1\le k\le K$ and $\gcd (l,k)=1$. Set
\begin{equation} \label{gdef}
g(x) = -tx^{\gamma},
\end{equation}
where $1/2 < \gamma < 1$.
Define positive real numbers $M_0$, $M_1$ and $M_2$ by
\begin{equation} \label{m1m2def}
g'(M_0) = \frac{l}{k}, \; g'(M_1) = \frac{l}{k} - \frac{A_1}{kK} \; \mbox{and} \; g'(M_2) = \frac{l}{k} + \frac{A_2}{kK}
\end{equation}
and assume that $2 \leq M_1 < M_2 \leq 2M_1$.  We further assume that
\begin{equation} \label{kcond}
 1 \leq k \leq K \ll M_1^{1/2-\delta_1}
\end{equation}
and
\begin{equation} \label{m1cond}
\frac{M_1^{1-\gamma+\delta_3}}{t} \ll kK \ll \frac{M_1^{1-\gamma/2-\delta_2}}{t^{1/2}} \; \mbox{and} \; k^3K \ll \frac{M_1^{3-2\gamma-\delta_2}}{t^2}.
\end{equation}
Then we have
\begin{equation} \label{transformlemmaeq}
\begin{split}
&\sum_{M_1 \leq m \leq M_2} a(m)e(-tm^{\gamma}) \\
&= i \left( \frac{t\gamma}{|l|/k}\right)^{\frac{\kappa-1/2}{2(1-\gamma)}} \sqrt{\frac{1}{2|l|(1-\gamma)}} \sum_{j=1}^2 (-1)^{j-1}
\sum_{n < A_j^2M_0/K^2} a(n) n^{-\kappa/2+1/4} e\left(
F_{j,k,l,t}(n)\right) \\
& \hspace*{2cm} + O \left( M_1^{\varepsilon}\left(\frac{M_1^{\kappa/2+2-3\gamma/2}}{t^{3/2}k^{3/2}K^{5/2}}+ \sqrt{\frac{k}{K}} M_1^{\kappa/2}+ M_1^{\kappa/2-1/4}k^{3/4}K^{1/4} \right)\right),
\end{split}
\end{equation}
where $F_{j,k,l,t}$ is a twice-differentiable function on the interval $[1,A_j^2M_0K^{-2}]$ whose second derivative satisfies the asymptotic estimate
\begin{equation} \label{secondderivative}
 F_{j,k,l,t}''(x)= \frac{(-1)^{j}}{2x^{3/2}k}\left(\frac{t\gamma}{|l|/k}\right)^{\frac{1}{2(1-\gamma)}}\left(1+ O\left(\frac{M_1^{2(1-\gamma)}}{t^2k^2K^2}\right)\right).
\end{equation}
\end{lemma}

\begin{proof}  We apply Lemma~\ref{expsumtransformation} with
\[ g(z) = -tz^{\gamma} = -t \exp \left( \gamma \log z \right), \]
$\log z$ being the principal branch of the logarithm, $M_1$ and $M_2$ defined in \eqref{m1m2def} and $w(z) =1$.  We may set
\begin{equation} \label{wgdef}
 W=1, \; G = tM_1^{\gamma}.
\end{equation}
so that the conditions in \eqref{wgcond} and \eqref{gprprcond} are satisfied (note that for \eqref{gprprcond} to be satisfied, the negative sign in the definition of $g(z)$ is necessary).  In the following, we check that the remaining conditions in Lemma~\ref{expsumtransformation} are satisfied following the notations of the same.  We set $r=l/k$ and note that
\[ |r| \asymp GM_1^{-1} \]
by \eqref{m1m2def}.  For
\[ M_0 = \left( \frac{t \gamma k}{|l|} \right)^{\frac{1}{1-\gamma}} \in (M_1, M_2), \]
we have $g'(M_0)=r$. Hence, \eqref{otto} is satisfied. By \eqref{m1m2def}, we have
\[ M_1 = \left( \frac{t \gamma}{|l|/k+A_1/(kK)} \right)^{\frac{1}{1-\gamma}}, \; \mbox{and} \; M_2 = \left( \frac{t \gamma}{|l|/k-A_2/(kK)} \right)^{\frac{1}{1-\gamma}}. \]
Using Taylor's theorem from differential calculus, we have the following estimate for $m_1$ as defined in \eqref{mjdef}.
\[ m_1 = \left( \frac{t\gamma}{|l|/k} \right)^{\frac{1}{1-\gamma}} - \left( \frac{t \gamma}{|l|/k+A_1/(kK)} \right)^{\frac{1}{1-\gamma}} \asymp \frac{A_1}{kK} \left( \frac{t\gamma}{|l|/k}\right)^{{\frac{1}{1-\gamma}}}\frac{1}{|l|/k}. \]
We further estimate the above expression using $A_1 \asymp 1$ and
\begin{equation} \label{rest}
 |r| = \frac{|l|}{k} \asymp t M_1^{\gamma-1}
\end{equation}
and get
\begin{equation} \label{m1est}
 m_1 \asymp \frac{M_1^{2-\gamma}}{tkK}.
\end{equation}
Similarly, we get
\begin{equation} \label{m2est}
 m_2 \asymp \frac{M_1^{2-\gamma}}{tkK}.
\end{equation}
By a short calculation, we see that the condition in \eqref{m1cond} is equivalent to that in \eqref{M1M2cond} in Lemma~\ref{expsumtransformation} by the virtue of \eqref{rest} and \eqref{m1est}.  Hence, all conditions in Lemma~\ref{expsumtransformation} are satisfied. \newline

Following the notations of Lemma~\ref{expsumtransformation}, we further have for $j=1,2$
\begin{equation} \label{pjndef}
 p_{j,n}(x) = -tx^{\gamma} - \frac{l}{k} x + (-1)^{j-1} \left( \frac{2\sqrt{nx}}{k} - \frac{1}{8} \right)
\end{equation}
and
\begin{equation} \label{njdef}
 n_j = \left( \frac{l}{k} - g'(M_j) \right)^2 k^2 M_j = \left( \frac{A_j}{kK} \right)^2 k^2 M_j = \frac{A_j^2}{K^2} M_j.
\end{equation}
To establish \eqref{transformlemmaeq}, we now approximate the terms in the exponential sum appearing on the right-hand side of \eqref{jutilaeq}, where we assume that $n < n_j$.  As in Lemma \ref{expsumtransformation}, we denote by $x_{j,n}$ the unique zero of $p'_{j,n}(x)$ in the interval $(M_1,M_2)$. We first observe that
\begin{equation} \label{firstobserve}
x_{j,n} - x_{j,0} \ll M_2 - M_1 = m_1 + m_2 \ll \frac{M_1^{2-\gamma}}{tkK}
\end{equation}
by \eqref{m1est} and \eqref{m2est}.  Hence
\begin{equation} \label{secondobserve1}
x_{j,n} = x_{j,0} \left( 1 + O \left( \frac{M_1^{1-\gamma}}{kKt} \right) \right).
\end{equation}
Moreover,
\begin{equation} \label{x0}
x_{j,0}= \left( \frac{t\gamma}{|l|/k} \right)^{\frac{1}{1-\gamma}} = M_0.
\end{equation}
Hence, we have
\begin{equation} \label{secondobserve2}
x_{j,n}^{\kappa/2-3/4} = \left( \frac{t\gamma}{|l|/k} \right)^{\frac{\kappa/2-3/4}{1-\gamma}} \left( 1 + O \left( \frac{M_1^{1-\gamma}}{kKt} \right) \right)
\end{equation}
by Taylor's theorem, \eqref{secondobserve1} and \eqref{x0}.  Furthermore, using Taylor's theorem again, \eqref{m1cond}, \eqref{njdef}, \eqref{secondobserve1} and \eqref{x0}, we obtain
\begin{eqnarray} \label{pr}
p_{j,n}''(x_{j,n})&=& p_{j,n}''(x_{j,0})\left(1+O\left(\frac{M_1^{1-\gamma}}{kKt} \right) \right)\\
&=& \left(t\gamma(1-\gamma)x_{j,0}^{\gamma-2}+(-1)^j\frac{\sqrt{n}}{2x_{j,0}^{3/2}k}\right)\left(1+O\left(\frac{M_1^{1-\gamma}}{kKt} \right) \right)\nonumber\\
&=& t\gamma(1-\gamma)x_{j,0}^{\gamma-2}\left(1+O\left(\frac{M_1^{1-\gamma}}{kKt} \right) \right)^2\nonumber\\
&=& t\gamma(1-\gamma)x_{j,0}^{\gamma-2}\left(1+O\left(\frac{M_1^{1-\gamma}}{kKt} \right) \right)\nonumber\\
&=&  (t\gamma)^{\frac{-1}{1-\gamma}}(1-\gamma)\left(\frac{|l|}{k}\right)^{\frac{2-\gamma}{1-\gamma}}\left(1+O\left(\frac{M_1^{1-\gamma}}{kKt} \right) \right)\nonumber
\end{eqnarray}
which implies that
\begin{equation} \label{prpr}
p_{j,n}''(x_{j,n})^{-1/2}= (t\gamma)^{\frac{1}{2(1-\gamma)}}(1-\gamma)^{-1/2}\left(\frac{|l|}{k}\right)^{\frac{\gamma-2}{2(1-\gamma)}}\left(1+O\left(\frac{M_1^{1-\gamma}}{kKt} \right) \right).
\end{equation}

Furthermore, $n_j$ defined in \eqref{njdef} can be approximated in the following way.
\begin{equation} \label{njapprox}
 n_j = \frac{A_j^2}{K^2} M_0 + O \left( \frac{M_1^{2-\gamma}}{tkK^3} \right)=
\frac{A_j^2}{K^2} M_0\left(1+O\left(\frac{M_1^{1-\gamma}}{kKt} \right) \right),
\end{equation}
where we have used \eqref{firstobserve}. \newline

Moreover, we write
\begin{equation} \label{Ffunction}
F_{j,k,l,t}(n)=p_{j,n}(x_{j,n})+\frac{1}{8}-n\cdot \frac{\overline{l}}{k}.
\end{equation}

Now applying \eqref{jutilaeq} and using \eqref{secondobserve2}, \eqref{prpr}, \eqref{njapprox}, \eqref{Ffunction}, \eqref{wgdef}, \eqref{rest}, \eqref{m1est} and the fact that $a(n) \ll n^{\kappa/2-1/2+\varepsilon}$, we get the asymptotic estimate \eqref{transformlemmaeq}. \newline

The remaining task is to prove the asymptotic estimate \eqref{secondderivative}.  Like in \cite{MJ}, we interpret $n$ in \eqref{Ffunction}, for a moment, as a continuous variable and aim to approximate the second derivative of $F_{j,k,l,t}(n)$ with respect to $n$.  We cannot directly use the approximation obtained in \cite{MJ} since it turns out to be not sufficient for our purposes. In the following, we refine Jutila's treatment of the said second derivative by evaluating the terms appearing in his method more precisely. \newline

Similarly as on page 107 in \cite{MJ}, we have
\[
\frac{\dif F_{j,k,l,t}(n)}{\dif n}=-\frac{\overline{l}}{k}+p'_{j,n}(x_{j,n})\frac{\dif x_{j,n}}{\dif n}+(-1)^{j-1}k^{-1}n^{-1/2}x_{j,n}^{1/2}=
-\frac{\overline{l}}{k}+(-1)^{j-1}k^{-1}n^{-1/2}x_{j,n}^{1/2},
\]
and further (compare with (4.3.28) on page 107 in \cite{MJ})
\[
\frac{\dif^2 F_{j,k,l,t}(n)}{\dif n^2}=(-1)^{j-1}\frac{1}{2}k^{-1}n^{-1/2}x_{j,n}^{-1/2}\frac{\dif x_{j,n}}{\dif n}-(-1)^{j-1}\frac{1}{2}k^{-1}n^{-3/2}x_{j,n}^{1/2}.
\]
We now express $\dif x_{j,n}/\dif n$ explicitly. We have
\[ p_{j,n}'(x) = -t\gamma x^{\gamma-1} - \frac{l}{k} + (-1)^{j-1} \frac{\sqrt{n}}{\sqrt{x}k} \]
and hence, by the definition of $x_{j,n}$,
\begin{equation} \label{solveeq}
 f(x_{j,n},n) =0,
\end{equation}
where
$$
f(x,n)=t\gamma k x^{\gamma-1/2} - |l| \sqrt{x} - (-1)^{j-1} \sqrt{n}.
$$
By implicit differentiation, we thus get
\begin{equation} \label{implicit}
\frac{\dif x_{j,n}}{\dif n}=-\frac{f_n(x_{j,n},n)}{f_x(x_{j,n},n)}=
(-1)^{j-1}\frac{\sqrt{x_{j,n}}}{\sqrt{n}\left(t\gamma(2\gamma-1)kx_{j,n}^{\gamma-1}-|l|\right)}.
\end{equation}
Hence,
\begin{equation} \label{Fsecondderivative}
\frac{\dif^2 F_{j,k,l,t}(n)}{\dif n^2}=\frac{1}{2kn\left(t\gamma(2\gamma-1)kx_{j,n}^{\gamma-1}-|l|\right)}-(-1)^{j-1}\frac{1}{2}k^{-1}n^{-3/2}x_{j,n}^{1/2}.
\end{equation}

We now approximate the terms on the right-hand side of \eqref{Fsecondderivative}. Using Taylor's formula together with \eqref{m1cond}, \eqref{rest}, \eqref{secondobserve1} and \eqref{x0}, we have the following asymptotic estimate for the first term.
\begin{eqnarray} \label{firstterm}
\frac{1}{2kn\left(t\gamma(2\gamma-1)kx_{j,n}^{\gamma-1}-|l|\right)}&=&
\frac{1}{2kn\left(t\gamma(2\gamma-1)kx_{j,0}^{\gamma-1}-|l|\right)}\left(1+O \left( \frac{M_1^{1-\gamma}}{kKt} \right)\right)\\
&=&\frac{1}{4(\gamma-1)k|l|n}\left(1+O \left( \frac{M_1^{1-\gamma}}{kKt} \right)\right)\nonumber\\ &=&
\frac{1}{4(\gamma-1)k|l|n}+O \left( \frac{M_1^{2-2\gamma}}{t^2k^3Kn}\right).\nonumber
\end{eqnarray}
For the approximation of the second term, we will need a better approximation than \eqref{secondobserve1} for $x_{j,n}$. We set
\[ g(x) = t\gamma k x^{\gamma-1/2} -|l| \sqrt{x}. \]
Then by Taylor's theorem, \eqref{solveeq} and $g(x_{j,0})=0$, we have
\[ (-1)^{j-1} \sqrt{n} = g(x_{j,n}) = g'(x_{j,0})(x_{j,n}-x_{j,0}) + O \left( \sup_{x \in (M_1,M_2)}|g''(x)| (M_2-M_1)^2 \right) \]
which implies that
\[
 x_{j,n}-x_{j,0} = (-1)^{j-1} \frac{\sqrt{n}}{g'(x_{j,0})} + O \left( M_1^{-1} (M_2-M_1)^2 \right).
\]
From the above, using \eqref{firstobserve} and \eqref{x0}, we obtain
\begin{equation} \label{ucomp}
 x_{j,n}-x_{j,0} = (-1)^{j-1} \frac{\sqrt{nx_{j,0}}}{(\gamma-1)|l|} + O\left(\frac{M_1^{3-2\gamma}}{t^2k^2K^2}\right).
\end{equation}
Now using Taylor's formula, \eqref{firstobserve} and \eqref{ucomp}, we obtain the following asymptotic estimate for the second term on the right-hand side of \eqref{Fsecondderivative}.
\begin{equation} \label{secondterm}
\begin{split}
(-1)^{j-1} \frac{1}{2}k^{-1}n^{-3/2}x_{j,n}^{1/2} &=(-1)^{j-1}\frac{1}{2}k^{-1}n^{-3/2}x_{j,0}^{1/2}+(-1)^{j-1}\frac{x_{j,n}-x_{j,0}}{4kn^{3/2}x_{j,0}^{1/2}}+
O\left(\frac{(M_2-M_1)^2}{kn^{3/2}x_{j,0}^{3/2}}\right) \\
&= (-1)^{j-1}\frac{1}{2}k^{-1}n^{-3/2}x_{j,0}^{1/2}+\frac{1}{4k|l|(\gamma-1)n} + O\left(\frac{M_1^{5/2-2\gamma}}{t^2k^3K^2n^{3/2}}\right).
\end{split}
\end{equation}
We note that the error term in \eqref{firstterm} can be absorbed into the error term in \eqref{secondterm} since $n\ll M_1/K^2$. Now
combining \eqref{x0}, \eqref{Fsecondderivative}, \eqref{firstterm} and \eqref{secondterm}, we get the relation \eqref{secondderivative}. This completes the proof.
\end{proof}

\section{Reduction to multi-linear sums}

In this section, we transform the sum $K^*_d$ appearing in \eqref{Kstareq} into a multi-linear exponential sum with monomials.  Let
\begin{equation} \label{KstardQdef}
 K^*_d (Q) = \sum_{h \sim H} \sum_{m \sim X/d} C_h A_m \sum_{k \sim Q} \sum_{\substack{l<0\\ \gcd(l,k) =1 \\ |l| \in \mathcal{L}(d,h,m,k) }} \sum_{n \in \mathcal{J}(d,h,m;l/k)} a(n) e (-hd^{2\gamma} m^{\gamma} n^{\gamma})
\end{equation}
be the contribution to $K^*_d$ of the terms with $k \sim Q$.  To prove \eqref{kstartdgoal}, it suffices to show that
\begin{equation} \label{KstardQgoal}
 K^*_d(Q) \ll \frac{N^{1-4\eta}}{d} \left( \frac{Y}{d} \right)^{(\kappa-1)/2},
\end{equation}
for $1 \leq Q \leq K$. A short computation using \eqref{Jlength} and $XY=N$ gives that the trivial bound for $K^*_d (Q)$ is
\[ K^*_d (Q) \ll \frac{Q}{K} \frac{N^{1+\varepsilon}}{d^2} H \left( \frac{Y}{d} \right)^{(\kappa-1)/2}. \]
Hence it is enough to prove \eqref{KstardQgoal} for
\begin{equation} \label{Qrange}
 N^{-5\eta} \frac{dK}{H} \ll Q \ll K.
\end{equation}

In this case, we transform the inner-most sum of \eqref{KstardQdef} using Lemma~\ref{transformlemma} into
\begin{equation} \label{transform1}
\begin{split}
&\sum_{n \in \mathcal{J}(d,h,m;l/k)} a(n)e(-hd^{2\gamma}m^{\gamma}n^{\gamma}) \\
&= i \left( \frac{hd^{2\gamma}m^{\gamma} \gamma}{|l|/k}\right)^{\frac{\kappa-1/2}{2(1-\gamma)}} \sqrt{\frac{1}{2|l|(1-\gamma)}} \sum_{j=1}^2 (-1)^{j-1}
\sum_{n < \mathcal{N}_j(d,h,m;l/k)} a(n) n^{-\kappa/2+1/4} e\left(F_{j,k,|l|,h,m,d}(n)\right) \\
& \hspace*{1cm} + O \left( N^{\varepsilon} \left( \frac{(Y/d)^{\kappa/2+2}}{H^{3/2}N^{3\gamma/2}Q^{3/2}K^{5/2}}+ \sqrt{\frac{Q}{K}} \left( \frac{Y}{d} \right)^{\kappa/2}+ \left( \frac{Y}{d} \right)^{\kappa/2-1/4}Q^{3/4}K^{1/4} \right) \right),
\end{split}
\end{equation}
where
\[ \mathcal{N}_j \left( d,h,m; \frac{l}{k} \right) = \frac{A_j(l/k)^2}{K^2} \left( \frac{hd^{2\gamma}m^{\gamma} \gamma}{|l|/k}\right)^{\frac{1}{1-\gamma}} \]
and the second derivative of the function $F_{j,k,|l|,h,m,d}$ satisfies the estimate
\begin{equation} \label{secder}
F_{j,k,|l|,h,m,d}''(x)= \frac{(-1)^{j}}{2x^{3/2}k}\left(\frac{hd^{2\gamma}m^{\gamma}\gamma}{|l|/k}\right)^{\frac{1}{2(1-\gamma)}}\left(1+
O\left(\left(\frac{Y}{HdN^{\gamma}QK}\right)^2 \right)\right).
\end{equation}
We also require that the conditions in \eqref{kcond} and \eqref{m1cond} are satisfied.  It is easy to check, using \eqref{Qrange}, that this is the case if the following condition holds.
\begin{equation} \label{Kcondition2}
 N^{6\eta} \frac{Y^{1/2}}{dN^{\gamma/2}} \ll K \ll N^{-6\eta} \min \left\{ \frac{Y^{1/2}}{H^{1/4}d^{1/2}N^{\gamma/4}}, \frac{Y^{3/4}}{H^{1/2}d^{3/4} N^{\gamma/2}} \right\}.
\end{equation}

We now insert \eqref{transform1} into \eqref{KstardQdef}.  The contribution to $K^*_d(Q)$ of the $O$-terms in \eqref{transform1} is
\[
 \ll N^{\varepsilon} E_d \left( \frac{Y}{d} \right)^{\kappa/2-1/2}
\]
with
\begin{equation}  \label{EdQdef}
 E_d := \frac{H^{1/2} N^{1-\gamma/2} Y^{1/2}}{K^2d^{5/2}} + \frac{H^2 N^{1+\gamma} K^2}{Y^{3/2} d^{1/2}} + \frac{H^2 N^{1+\gamma} K^3}{ Y^{7/4} d^{1/4}},
\end{equation}
where we have used the facts $Q \ll K$ and $XY=N$. Hence, to establish \eqref{KstardQgoal}, we require that
\begin{equation} \label{Otermbound}
E_d\ll \frac{N^{1-5\eta}}{d}.
\end{equation}

The main term takes the form
\[ \sum_{h \sim H} \sum_{m \sim X/d} \sum_{k \sim Q} \sum_{\substack{l<0\\ \gcd(l,k) =1 \\ |l| \in \mathcal{L}(d,h,m,k) }} \sum_{j=1}^2 \ \sum_{n < \mathcal{N}_j(d,h,m;l/k)} \cdots . \]
We make the summation ranges for $l$ and $n$ independent of the other variables by using Perron's formula, Lemma~\ref{Perron}, several times.  This treatment is for convenience rather than a matter of nessessity, as the application of Perron's formula enables us to avoid some summation conditions which one would otherwise encounter.  We note that the contribution arising from the error term in this treatment, the $O$-term in Perron's formula \eqref{perroneq}, is negligible.  Then after re-arranging the order of summations, breaking the $l$-range into dyadic intervals and estimating the sizes of the coefficients (where we use the Ramanujan-Petersson bound), it suffices to show that
\begin{equation} \label{whatwewant}
\frac{Y^{3/4}}{H^{1/2}d^{3/4}N^{\gamma/2}Q^{1/2}} T_{j,d}^{\pm}(Q) \ll \frac{N^{1-5\eta}}{d}
\end{equation}
with
\begin{equation} \label{quadlinesum}
T_{j,d}(Q):=\sum_{n < c_1Y/(K^2d)} n^{-1/4} \sum_{k \sim Q}\left| \sum_{h \sim H} \ \sum_{m \sim X/d} \ \sum_{\substack{l \sim c_2 L \\ \gcd (l,k)=1}}  \psi_{k,l,h,m} e \left(F_{j,k,l,h,m,d}(n) \right) \right|
\end{equation}
in order to establish \eqref{KstardQgoal}, where $j=1,2$,
$|\psi_{k,l,h,m}| \leq 1$, $c_1$ and $c_2$ are positive constants of bounded size, and
\begin{equation} \label{Ldef2}
L:=HdN^{\gamma}Y^{-1}Q.
\end{equation}

\section{Estimation of the multi-linear sums}

We now estimate the multi-linear sum $T_{j,d}(Q)$ defined in \eqref{quadlinesum}. It will suffice to consider the case $j=2$. The treatment of the other case $j=1$ is similar.  First, we break the outer sum over $n$ into dyadic intervals and denote by $T_{2,d}(Q,R)$ the contribution to $T_{2,d}(Q)$ of the $n$'s with $n\sim R$, which we bound in the following. \newline

Applying the Cauchy-Schwarz inequality, and re-arranging the order of summation, we obtain
 \begin{equation} \label{afterCau}
T_{2,d}(Q,R)^2\ll R^{1/2} Q\sum_k \ \sum_{l,h,m} \ \sum_{\tilde l,\tilde h,\tilde m}  \left|
\sum_{n\sim R} e \left(F_{2,k,\tilde l,\tilde h,\tilde m,d}(n)- F_{2,k,l,h,m,d}(n)\right) \right|.
\end{equation}
We note that the $O$-term in \eqref{secder} equals $1/(KL)^2$. Hence,
\begin{equation} \label{secderbound}
F_{2,k,l,h,m,d}''(x)= \frac{1}{2x^{3/2}}\left(\frac{hd^{2\gamma}m^{\gamma}\gamma}{l/k}\right)^{\frac{1}{2(1-\gamma)}}\frac{1}{k}\left(1+O\left((KL)^{-2}\right)\right).
\end{equation}
Moreover, we have
\begin{equation} \label{Udef}
0< \left(\frac{hd^{2\gamma}m^{\gamma}\gamma}{l/k}\right)^{\frac{1}{2(1-\gamma)}}\frac{1}{k} \le c_3\frac{Y^{1/2}}{d^{1/2}Q}=:U
\end{equation}
for some constant $c_3>0$.  By $\mathcal{T}(k,\Delta)$ we denote the set of all six-tuples $(l,\tilde l,h,\tilde h,m,\tilde m)$ with $l,\tilde l\sim c_2L$, $\gcd (l,k)= \gcd (\tilde l,k) =1$, $h,\tilde h\sim H$ and $m,\tilde m\sim X/d$ such that
\[ \left| \left(\frac{\tilde h d^{2\gamma}\tilde{m}^{\gamma}\gamma}{\tilde l/ k}\right)^{\frac{1}{2(1-\gamma)}}\frac{1}{k} -
\left(\frac{hd^{2\gamma}m^{\gamma}\gamma}{l/k}\right)^{\frac{1}{2(1-\gamma)}}\frac{1}{k} \right|
\le \Delta.\]

By the following lemma, $\mathcal{T}(k,\Delta)$ has essentially the expected cardinality.

\begin{lemma} \label{cardlemma} Let $0\le \Delta\le U$. Then
\[ |\mathcal{T}(k,\Delta)| \ll N^{\epsilon}\left(LHXd^{-1}+\frac{\Delta}{U} L^2H^2X^2d^{-2}\right). \]
\end{lemma}

We postpone the proof of Lemma \ref{cardlemma} to the next section.
We now set
\[ \Delta_0:= N^{\eta}U\left(\frac{1}{(KL)^2}+\frac{d}{LHX}\right). \]
Then Lemma \ref{cardlemma} implies that
\begin{equation} \label{Tcardbound}
|\mathcal{T}(k,\Delta)| \ll
N^{\eta}\frac{\Delta}{U} L^2H^2X^2d^{-2} \mbox{ if } \Delta\ge \Delta_0.
\end{equation}
We further set
\[ \mathcal{T}'(k,\Delta)=\mathcal{T}(k,\Delta)\setminus \mathcal{T}(k,\Delta/2).\]
Then from \eqref{afterCau}, we deduce that
\begin{equation} \label{afterCaureduce}
T_{2,d}(Q,R)^2\ll R^{1/2}Q\sum_{k\sim Q}\left(\Sigma_1(k)+(\log N)\max_{\Delta_0\le \Delta\le U}\Sigma_2(k,\Delta)\right),
\end{equation}
where
\begin{equation} \label{Sigma1}
\Sigma_1(k):=\sum_{(l,\tilde l,h,\tilde h,m,\tilde m) \in \mathcal{T}(k,\Delta_0)}  \left|
\sum_{n\sim R} e \left(F_{2,k,\tilde l,\tilde h,\tilde m,d}(n)- F_{2,k,l,h,m,d}(n)\right) \right|
\end{equation}
and
\begin{equation} \label{Sigma2}
\Sigma_2(k,\Delta):=\sum_{(l,\tilde l,h,\tilde h,m,\tilde m) \in \mathcal{T}'(k,\Delta)}  \left|
\sum_{n\sim R} e \left(F_{2,k,\tilde l,\tilde h,\tilde m,d}(n)- F_{2,k,l,h,m,d}(n)\right) \right|.
\end{equation}

Using \eqref{Tcardbound}, we estimate $\Sigma_1$ trivially by
\begin{equation} \label{Sigma1bound}
\Sigma_1(k)\ll N^{2\eta}\left(RK^{-2}H^2X^2d^{-2}+RLHXd^{-1}\right).
\end{equation}
We now turn to the estimation of $\Sigma_2$. If $(l,\tilde l,h,\tilde h,m,\tilde m) \in \mathcal{T}'(k,\Delta)$, $\Delta\ge \Delta_0$ and $x\sim R$, then
\begin{equation} \label{diffrange}
\left|\frac{\dif^2}{\dif^2 x} \left( F_{2,k,\tilde l,\tilde h,\tilde m,d}(x)- F_{2,k,l,h,m,d}(x) \right) \right| \asymp
\frac{\Delta}{R^{3/2}}.
\end{equation}
Now \eqref{diffrange}, Lemma \ref{vandercorput} and \eqref{Tcardbound} yield the estimate
\begin{equation} \label{Sigma2bound}
\Sigma_2\ll N^{\eta}U^{-1}L^2H^2X^2d^{-2}\left(R^{1/4} \Delta^{3/2}+R^{3/4}\Delta^{1/2}\right).
\end{equation}
Combining \eqref{afterCaureduce}, \eqref{Sigma1bound} and \eqref{Sigma2bound}, we obtain
\begin{equation} \label{TQRbound}
T_{2,d}(Q,R)^2 \ll N^{2\eta}Q^2L^2H^2X^2d^{-2} \left(R^{3/2}K^{-2}L^{-2}+R^{3/2}L^{-1}H^{-1}X^{-1}d+R^{3/4}U^{1/2}+
R^{5/4}U^{-1/2}\right).
\end{equation}

Using $R\ll Y/(dK^2)$, $Q\ll K$, \eqref{Ldef2}, the definition of $U$ in \eqref{Udef}, and $XY=N$, we deduce from \eqref{TQRbound} that
\begin{equation} \label{TQbound}
\begin{split}
T_{2,d}(Q) \ll N^{\eta} \Big( QK^{-5/2}&Hd^{-7/4}NY^{-1/4} +Q^{3/2}K^{-3/2}Hd^{-3/4}N^{1/2+\gamma/2}Y^{-1/4}  \\
& + Q^{7/4}K^{-3/4}H^{2}d^{-1/2}N^{1+\gamma}Y^{-3/2} \Big).
\end{split}
\end{equation}
A similar estimate for $T_{1,d}(Q)$ can be established in essentially the same way. Hence, \eqref{whatwewant} holds if
\begin{equation} \label{whatwewant2}
\frac{H^{1/2}N^{1-\gamma/2}Y^{1/2}}{K^2d^{5/2}}+\frac{H^{1/2}N^{1/2}Y^{1/2}}{K^{1/2}d^{3/2}}+\frac{K^{1/2}H^{3/2}N^{1+\gamma/2}}{Y^{3/4}d^{5/4}}\ll \frac{N^{1-6\eta}}{d}.
\end{equation}

\section{Proof of the spacing lemma}

In this section, we provide a proof of Lemma~\ref{cardlemma}.

\begin{proof} (of Lemma~\ref{cardlemma}) Throughout, we assume that $l,\tilde l\sim c_2L$, $h,\tilde h\sim H$ and $m,\tilde m\sim X/d$ and ignore the condition $\gcd (l,k)= \gcd (\tilde l,k)=1$ in the definition of $\mathcal{T}(k,\Delta)$. \newline

We first observe that the inequality
\begin{equation} \label{spacing1}
\left(\frac{\tilde h d^{2\gamma}\tilde{m}^{\gamma}\gamma}{\tilde l/k}\right)^{\frac{1}{2(1-\gamma)}}\frac{1}{k} -
\left(\frac{hd^{2\gamma}m^{\gamma}\gamma}{l/k}\right)^{\frac{1}{2(1-\gamma)}}\frac{1}{k} \ll \Delta
\end{equation}
holds if
 \begin{equation} \label{spacing2}
\left(\frac{\tilde h d^{2\gamma}\tilde{m}^{\gamma}\gamma}{\tilde l/k}\right)^{\frac{1}{2(1-\gamma)}}\frac{1}{k}
\left(\left(\frac{hd^{2\gamma}m^{\gamma}\gamma}{l/k}\right)^{\frac{1}{2(1-\gamma)}}\frac{1}{k} \right)^{-1}-1 \ll \frac{\Delta}{U},
\end{equation}
with the implied $\ll$-constant in \eqref{spacing2} depending on the implied $\ll$-constant in \eqref{spacing1}. The inequality \eqref{spacing2} can be simplified into
\begin{equation} \label{spacing3}
\left(\frac{l\tilde h}{\tilde l h}\right)^{\frac{1}{2(1-\gamma)}} \left(\frac{\tilde{m}}{m}\right)^{\frac{\gamma}{2(1-\gamma)}}
-1 \ll \frac{\Delta}{U}
\end{equation}
which is satisfied if
\begin{equation} \label{spacing5}
\left(\frac{\tilde{m}}{m}\right)^{\frac{\gamma}{2(1-\gamma)}}-\left(\frac{\tilde l h}{l\tilde h}\right)^{\frac{1}{2(1-\gamma)}} \ll \frac{\Delta}{U},
\end{equation}
with the implied $\ll$-constant in \eqref{spacing5} depending on the implied $\ll$-constant in \eqref{spacing3}. Now the number of solutions to the inequality \eqref{spacing5} does not exceed the product of
$$\max\limits_{r\sim LH} d(r)^2$$
and the number of solutions to the inequality
\begin{equation} \label{spacing6}
\left(\frac{\tilde{m}}{m}\right)^{\frac{\gamma}{2(1-\gamma)}}-\left(\frac{\tilde r}{r}\right)^{\frac{1}{2(1-\gamma)}} \ll \frac{\Delta}{U},
\end{equation}
where $m,\tilde m\sim X/d$ and $r,\tilde r\asymp LH$.  Now using Lemma \ref{spacinglemma} and the bound $d(r) \ll r^{\varepsilon}$ for the divisor function, we obtain the desired result.
\end{proof}

\section{Estimation of $K_d$, second method}

We are now ready to formulate another condition under which the desired estimate $K_d\ll N^{1-3\eta}d^{-1}$ holds. This will be favorable in the situation when $Y$ is large.

\begin{lemma} \label{newestimateforKd}
We have
\begin{equation} \label{Kdinequality}
K_d \ll N^{1-3\eta}d^{-1},
\end{equation}
provided that $\gamma>1/2$, $1\le H\le N^{1-\gamma+\eta}$, $1\le d\le 2Y$ and
\begin{equation} \label{Kineq1}
\begin{split}
N^{26\eta}& \max\left\{1,\frac{H^{1/4}Y^{1/4}}{N^{\gamma/4}d^{3/4}},\frac{Y}{N^{\gamma}d}\right\}\\
& \ll \min\left\{\frac{Y^{1/2}}{H^{1/4}N^{\gamma/4}d^{1/2}}, \frac{Y^{3/4}}{H^{1/2} N^{\gamma/2}d^{3/4}}, \frac{Y^{3/4}}{HN^{\gamma/2}d^{1/4}}, \frac{Y^{7/12}}{H^{2/3}N^{\gamma/3}d^{1/4}},\frac{Y^{3/2}d^{1/2}}{H^3N^{\gamma}}\right\}.
\end{split}
\end{equation}
\end{lemma}

\begin{proof}
According to the results in sections 6 - 9, \eqref{Kdinequality} holds if there exists a real number $K\ge 1$ satisfying the conditions
\begin{equation} \label{Kcondition1}
 K \gg N^{4\eta} \frac{Y}{N^{\gamma}d^2},
\end{equation}
\begin{equation} \label{Kcondition3}
 N^{6\eta} \frac{Y^{1/2}}{dN^{\gamma/2}} \ll K \ll N^{-6\eta} \min \left\{ \frac{Y^{1/2}}{H^{1/4}d^{1/2}N^{\gamma/4}}, \frac{Y^{3/4}}{H^{1/2}d^{3/4} N^{\gamma/2}} \right\},
\end{equation}
\begin{equation}  \label{Otermbound1}
\frac{H^{1/2} N^{1-\gamma/2} Y^{1/2}}{K^2d^{5/2}} + \frac{H^2 N^{1+\gamma} K^2}{Y^{3/2} d^{1/2}} + \frac{H^2 N^{1+\gamma} K^3}{ Y^{7/4} d^{1/4}} \ll \frac{N^{1-5\eta}}{d}
\end{equation}
and
\begin{equation} \label{whatwewant3}
\frac{H^{1/2}N^{1-\gamma/2}Y^{1/2}}{K^2d^{5/2}}+\frac{H^{1/2}N^{1/2}Y^{1/2}}{K^{1/2}d^{3/2}}+\frac{K^{1/2}H^{3/2}N^{1+\gamma/2}}{Y^{3/4}d^{5/4}}\ll \frac{N^{1-6\eta}}{d}.
\end{equation}
The above inequalities \eqref{Kcondition1}, \eqref{Kcondition3}, \eqref{Otermbound1} and \eqref{whatwewant3} correspond to the inequalities \eqref{Kcondition}, \eqref{Kcondition2}, \eqref{Otermbound} and \eqref{whatwewant2}, respectively.
The terms in the minimum on the right-hand side of \eqref{Kcondition3}, the second and the third term on the left-hand side of \eqref{Otermbound1} and the last term on the left-hand side of \eqref{whatwewant3} lead to the condition
\begin{equation} \label{Bratwurst1}
K \ll  N^{-12\eta} \min \left\{\frac{Y^{1/2}}{H^{1/4}N^{\gamma/4}d^{1/2}}, \frac{Y^{3/4}}{H^{1/2} N^{\gamma/2}d^{3/4}}, \frac{Y^{3/4}}{HN^{\gamma/2}d^{1/4}}, \frac{Y^{7/12}}{H^{2/3}N^{\gamma/3}d^{1/4}},
\frac{Y^{3/2}d^{1/2}}{H^3N^{\gamma}}\right\}.
\end{equation}
The first term on the left-hand side of \eqref{Otermbound1} and the first two terms on the left-hand side of \eqref{whatwewant3} lead to the condition
\begin{equation} \label{Bratwurst2}
K\gg N^{12\eta}\max\left\{\frac{H^{1/4}Y^{1/4}}{N^{\gamma/4}d^{3/4}},\frac{HY}{Nd}\right\}.
\end{equation}
Using $H\le N^{1-\gamma+\eta}$, we observe that the lower bounds in \eqref{Kcondition1}, \eqref{Kcondition3} and \eqref{Bratwurst2} hold if
\begin{equation} \label{Bratwurst3}
K\gg N^{14\eta}\max\left\{1,\frac{H^{1/4}Y^{1/4}}{N^{\gamma/4}d^{3/4}},\frac{Y}{N^{\gamma}d}\right\}.
\end{equation}
Obviously, a number $K\ge 1$ satisfying \eqref{Bratwurst1} and \eqref{Bratwurst3} exists if \eqref{Kineq1} holds. This completes the proof.
\end{proof}

Lemma \ref{newestimateforKd} can be simplified into the following.

\begin{lemma} \label{Kdcondlemma2}
For every sufficiently small fixed $\eta>0$, we have
\[ K_d\ll N^{1-3\eta}d^{-1}, \]
provided that $\gamma>13/16$, $1\le H\le N^{1-\gamma+\eta}$, $1\le d\le 2Y$ and
$$
N^{\max\{2-4\gamma/3,13/5-2\gamma,6-6\gamma\}+1000\eta}\le Y\le 2N.
$$
\end{lemma}

\begin{proof}
If $d\ge N^{3\eta}H$, then the desired estimate $K_d\ll N^{1-3\eta}d^{-1}$ follows from the trivial estimate $K_d\ll HXY/d^2= HN/d^2$. If $d\le N^{3\eta}H\le N^{1-\gamma+4\eta}$ and $\gamma>13/16$, then we use Lemma \ref{newestimateforKd}. In this case we calculate that \eqref{Kineq1} and hence $K_d\ll N^{1-3\eta}d^{-1}$ holds if
\begin{equation}
N^{\max\left\{\frac{5}{3}-\gamma,
2-\frac{4\gamma}{3},\frac{11}{4}-\frac{5\gamma}{2},\frac{13}{5}-2\gamma,6-6\gamma\right\}+1000\eta}\le Y\le 2N.
\end{equation}
The first term in the maximum is dominated by the second term if $\gamma<1$, and the third term is dominated by the fourth term if $\gamma>3/10$. This completes the proof.
\end{proof}

Combining the above Lemma \ref{Kdcondlemma2} with Lemmas \ref{Kdcondlemma0} and \ref{Kdcondlemma}, we arrive at the following conclusion.

\begin{lemma} \label{finalKdcondlemma}
For every sufficiently small fixed $\eta>0$, we have
\[ K_d\ll N^{1-3\eta}d^{-1}, \]
provided that $\gamma>7/8$, $1\le H\le N^{1-\gamma+\eta}$, $1\le d\le 2Y$ and
$$
N^{3-3\gamma+100\eta}\le Y\le 2N.
$$
\end{lemma}

\begin{proof}
This follows from Lemmas \ref{Kdcondlemma0}, \ref{Kdcondlemma} and \ref{Kdcondlemma2} upon noting that $\gamma>2-4\gamma/3$ if $\gamma>6/7$, $\gamma>13/5-2\gamma$ if $\gamma>13/15$, $\gamma>6-6\gamma$ if $\gamma>6/7$, and $3\gamma-2>5-5\gamma$ if $\gamma>7/8$.
\end{proof}

\section{Proof of the main result \label{finalproof}}

\begin{proof} (of Theorems~\ref{mainresultmodified} and~\ref{mainresult}) We recall that Theorem~\ref{mainresultmodified} and hence Theorem~\ref{mainresult}, our main result, holds if \eqref{goal1} is valid for any $N\ge 1$ and $1\le H\le N^{1-\gamma+\eta}$. Here $\gamma$ is a fixed number in the range $7/8<\gamma<1$, and $\eta$ is sufficiently small, which we assume in the following. Furthermore, in Lemma \ref{bilinearsums} we formulated some conditions on bilinear sums $K_d$ and $L_d$ under which \eqref{goal1} holds. In the following, we check that these conditions are satisfied. \newline

We choose the parameters $u$, $v$ and $z$ in Lemma \ref{bilinearsums} as follows.
\begin{eqnarray*}
u&:=&N^{1-\gamma+100\eta},\\
v&:=&4N^{1/3},\\
z&:=&\left[N^{\gamma/2-100\eta}\right]+1/2.
\end{eqnarray*}
The parameters $u$, $v$ and $z$, so chosen, indeed satisfy the conditions in \eqref{uvzcond} if $\gamma>4/5$ and $\eta$ is sufficiently small. Moreover, the conditions \eqref{Kdbound} and \eqref{Ldbound} hold by Lemmas \ref{Ldcondlemma} and \ref{finalKdcondlemma} since $5\gamma-4>1/3$ if $\gamma>13/15$ (the exponent which marks the limit of the method of Liu-Rivat in \cite{LiRi}) and $3-3\gamma<\gamma/2$ if $\gamma>6/7$. This completes the proof.
\end{proof}

\section{Notes}

The upper bound $c<8/7$ in Theorem \ref{mainresult} is due to the inequality $3\gamma-2>5-5\gamma$, which is equivalent to $\gamma>7/8$, in the proof of Lemma \ref{finalKdcondlemma}. If this inequality is satisfied, then the $Y$-ranges in Lemmas \ref{Kdcondlemma0} and \ref{Kdcondlemma} overlap which is required in order to establish the $Y$-range in Lemma \ref{finalKdcondlemma}. For the $Y$-ranges in Lemmas \ref{Kdcondlemma0} and \ref{Kdcondlemma2} to overlap, which is also required, we only need the weaker condition $\gamma>13/15$. In the proof of the main result in the previous section, the strongest condition occurring is also $\gamma>13/15$. We believe that the $c$-range in Theorem \ref{mainresult} could be slightly widened by modifying the method in \cite{RiSa} to obtain a better lower bound for $Y$ in Lemma \ref{Kdcondlemma0} (which does not depend on the Hecke eigenvalues). However, we have not tried to do so since the main focus of this paper lies on the treatment of the Hecke eigenvalues. \newline

We further note that an improvement of the lower bound in Lemma \ref{Kdcondlemma0} would also correspond to a better upper bound for $Y$ in Lemma \ref{Ldcondlemma}, which would lead to a weakening of the condition $\gamma>13/15$ occurring in section \ref{finalproof}. However, this would be less significant since the condition $\gamma>13/15$ already occurs in the proof of Lemma \ref{finalKdcondlemma} and seems difficult to improve at this place. We point out that the last-mentioned condition is due to to the inequality $\gamma>13/5-2\gamma$ coming from the upper bound for $Y$ in Lemma \ref{Kdcondlemma0}, which is likely to be best possible, and from the lower bound for $Y$ in Lemma \ref{Kdcondlemma2}, which depends on our treatment of the Hecke eigenvalues. \newline

Besides slight improvements, it would be highly desirable to prove Conjecture~\ref{conj} for some $c$-range, making Theorem \ref{oscillations} unconditional for the same $c$-range.  To this end, we need estimates for exponential sums with squares of Hecke eigenvalues (or more generally, with Fourier coefficients of Rankin-Selberg convolutions of cusp forms). \newline

Finally, it would be interesting to generalize our result to cusp forms of arbitrary level. One would need to work out a generalization of Jutila's method to arbitrary levels for this purpose. \newline

\noindent{\bf Acknowledgments.} This project was started when S. B. visited the Division of Mathematical Sciences of Nanyang Technological University (NTU).  He wishes to thank the Division for its generous financial supports and warm hospitality.  S. B. further wishes to thank Jacobs University for providing excellent working conditions.  L. Z. was supported by an Academic Research Fund Tier 1 Grant at NTU during this work.

\bibliography{biblio}
\bibliographystyle{amsxport}

\vspace*{.5cm}

\noindent\begin{tabular}{p{8cm}p{8cm}}
School of Engineering \& Science, Jacobs Univ. & Div. of Math. Sci., School of Phys. \& Math. Sci., \\
P. O. Box 750561, 28725 Bremen, Germany & Nanyang Technological Univ., 637371 Singapore \\
Email: {\tt s.baier@jacobs-university.de} & Email: {\tt lzhao@pmail.ntu.edu.sg} \\
\end{tabular}

\end{document}